\documentclass[12pt]{article}

\usepackage[utf8]{inputenc}
\usepackage{amsmath,amsthm}
\usepackage{bbm}
\usepackage{framed}
\usepackage{color}
\usepackage{xcolor}
\usepackage{hyperref}
\hypersetup{
    colorlinks=true,
    urlcolor = {darkred},
    linkcolor = {darkred},
    citecolor = {darkred},
    linkcolor = {darkred},
    linktoc=all
}
\usepackage{tikz}
\usepackage{tikz-cd}
\usepackage{tikz-3dplot} 
\usepackage{xifthen} 
\usepackage[inline]{enumitem}
\usepackage{mathtools} 
\usepackage{bbm} 
\makeatletter
\newcommand*{\rom}[1]{\expandafter\@slowromancap\romannumeral #1@}
\makeatother

\usepackage{amssymb}
\usepackage{fourier}
\usepackage{graphicx}
\usepackage{tikz}
\usepackage{tikz-cd}
\usepackage[inline]{enumitem}
\usepackage{extpfeil}
\usepackage{soul}
\usepackage{authblk}
\usepackage[margin=1in]{geometry}

\usetikzlibrary{matrix,arrows,decorations.pathmorphing,patterns,shapes.geometric}

\newtheorem{theorem}{Theorem}[section]

\theoremstyle{definition}
\newtheorem{notation}[theorem]{Notation}
\newtheorem{proposition}[theorem]{Proposition}
\newtheorem{remark}[theorem]{Remark}
\newtheorem{definition}[theorem]{Definition}
\newtheorem{lemma}[theorem]{Lemma}
\newtheorem{corollary}[theorem]{Corollary}

\newtheorem{example}[theorem]{Example}
\definecolor{darkblue}{rgb}{0.0, 0.0, 0.8}
\definecolor{darkred}{rgb}{0.8, 0.0, 0.0}
\definecolor{darkgreen}{rgb}{0.0, 0.5, 0.0}
\definecolor{darkgray}{rgb}{0.66, 0.66, 0.66}

\newcommand{\one}{\mathbbm{1}}
\newcommand{\cov}{\mathrm{cov}}

\newcommand{\con}{\mathbf{Con}}
\newcommand{\conop}{\con^{\mathrm{op}}}
\newcommand{\interop}{\inter^{\mathrm{op}}}
\newcommand{\rank}{\mathrm{rank}}

\newcommand{\inter}{\mathbf{Int}}
\newcommand{\niton}{\not\owns}

\newcommand{\ob}{\mathrm{ob}}

\newcommand{\Pb}{\mathbb{P}}

\newcommand{\QA}{\mathbb{A}}

\newcommand{\vect}{\mathbf{vec}}

\newcommand{\mult}{\mathrm{mult}}
\newcommand{\mrank}{\mathrm{multirk}}

\newcommand{\F}{k}

\newcommand{\Zb}{\mathbb{Z}}
\newcommand{\Zplus}{\mathbb{Z}_{\geq 0}}

\newcommand{\Rb}{\mathbb{R}}

\newcommand{\N}{\mathbb{N}}

\newcommand{\eps}{\varepsilon}

\newcommand{\dgm}{\mathrm{dgm}}
\newcommand{\intdgm}{\mathrm{dgm}_{\mathbb{I}}}

\newcommand{\abs}[1]{\left\lvert{#1}\right\rvert}

\newcommand{\rk}{\mathrm{rk}}
\newcommand{\mrk}{\mathrm{multirk}}

\newcommand{\barc}{\mathrm{barc}}

\newcommand{\Ccal}{\mathcal{C}}

\newcommand{\Ical}{\mathcal{I}}

\newcommand{\RNum}[1]{\uppercase\expandafter{\romannumeral #1\relax}}
\newcommand{\ba}{p}
\newcommand{\bb}{q}
\newcommand{\bc}{r}
\newcommand{\intrk} {\rk_{\mathbb{I}}}

\newcommand{\new}[1]                {{{#1}}}

\title{Bigraded Betti numbers and Generalized Persistence Diagrams}

\author[1]{Woojin Kim}
\author[2]{Samantha Moore}

\affil[1]{Department of Mathematical Sciences, KAIST, South Korea\thanks{\texttt{woojin.kim@kaist.ac.kr}}}
\affil[2]{Department of Mathematics, North Carolina School of Science and Mathematics- Morganton\thanks{\texttt{samantha.moore@ncssm.edu}}}

\begin{document}

\maketitle

\begin{abstract}Commutative diagrams of vector spaces and linear maps over $\Zb^2$ are objects of interest in topological data analysis (TDA) where  this type of  diagrams are called 2-parameter persistence modules. Given that quiver representation theory tells us that such diagrams are of wild type, studying informative invariants of a 2-parameter persistence module $M$ is of central importance in TDA. One of such invariants is the generalized rank invariant, recently introduced by Kim and M\'emoli.  Via the M\"obius inversion of the generalized rank invariant of $M$, we obtain a collection of connected subsets $I\subset\Zb^2$ with signed multiplicities. This collection generalizes the well known notion of \emph{persistence barcode} of a persistence module over $\mathbb{R}$ from TDA.  In this paper we show that the bigraded Betti numbers of $M$, a classical algebraic invariant of $M$, are obtained by counting the corner points of these subsets $I$s. Along the way, we verify that an invariant of 2-parameter persistence modules called the interval decomposable approximation (introduced by Asashiba et al.) also encodes the bigraded Betti numbers in a similar fashion. 
\new{We also show that the aforementioned results are optimal in the sense that they cannot be extended to $d$-parameter persistence modules for $d \geq 3$.}
\end{abstract}

\paragraph{keywords.} Multiparameter persistence, Multigraded Betti numbers, Quiver representations, M\"obius inversion, Persistent homology, Persistence diagram


\maketitle


\section{Introduction}

\paragraph{Multiparameter persistent homology.} 
Theoretical foundations of \emph{persistent homology}, one of the main protagonists in topological data analysis (TDA), have been rapidly developed in the last two decades, allowing  a large number of applications. Persistent homology is obtained by applying the homology functor to an $\Rb$ (or $\Zb$) -indexed increasing family of topological spaces \cite{carlsson2009topology,edelsbrunner2010computational}. This parametrized family of topological spaces, for example, often arises as either a \emph{sublevel set filtration} of a real-valued map on a topological space, or the \emph{Vietories-Rips simplicial filtration} of a metric space.

With more complex input data, we obtain $\Rb^d$-indexed increasing families ($d>1$) of topological spaces, e.g. a sublevel set filtration of a topological space that is filtered by multiple real-valued functions, or a Vietories-Rips-sublevel simplicial filtration of a metric space  equipped with a map \cite{carlsson2009topology,carlsson2009theory}. By applying the homology functor (with coefficients in a fixed field $\F$) to such a multiparameter filtration, we obtain a \emph{$d$-parameter persistence module} $\Rb^d\mbox{(or $\Zb^d$)}\rightarrow \vect$, a functor from the poset $\Rb^d\mbox{(or $\Zb^d$)}$ to the category $\vect$ of finite dimensional vector spaces and linear maps over the field $k$. 
In contrast to the case of $d=1$, there is no discrete and complete invariant  for $\Rb^d \mbox{(or $\Zb^d$)}\rightarrow \vect$ for $d>1$ \cite{carlsson2009theory}. In quiver representation theory, functors $\Rb^d\mbox{(or $\Zb^d$)}\rightarrow \vect$ ($d>1$) are of \emph{wild type}, implying that there is no simple invariant which completely encodes the isomorphism type of $\Rb^d\mbox{(or $\Zb^d$)}\rightarrow \vect$  \cite{derksen2005quiver,gabrielsthm}.
Nevertheless, there have been many studies on the invariants of $d$-parameter persistence modules, e.g.  \cite{carlsson2009theory,chacholski2017combinatorial,harrington2019stratifying,knudson2007refinement,miller2020homological,scolamiero2017multidimensional,vipond2020multiparameter}. 

Special attention has been placed on the case of $d=2$ \cite{asashiba2018interval,asashiba2019approximation,botnan2018algebraic,botnan2020rectangle,cochoy2020decomposition,dey2021rectangular,escolar2016persistence,lesnick2015interactive} in part because 2-parameter filtrations arise in the study of interlevel set persistence \cite{botnan2018algebraic,carlsson2009zigzag}, in the study of point cloud data with non-uniform density \cite{bauer2020cotorsion,cai2020elder,carlsson2010multiparameter,carlsson2009theory}, or in applications in material science and chemical engineering \cite{escolar2016persistence,hiraoka2016hierarchical,keller2018persistent}. The software RIVET \cite{lesnick2015interactive}  can efficiently compute and visualize the dimension function (a.k.a. the Hilbert function), the fibered barcode, and the bigraded Betti numbers of a 2-parameter persistence module.

\paragraph{Multigraded Betti numbers.} 

Multigraded Betti numbers encode important information about the algebraic structure of a multigraded module over the polynomial ring in $n$ variables \cite{eisenbudbetti,miller2005combinatorial}. For multiparameter persistence modules that arise from data, multigraded Betti numbers provide insight about the coarse-scale topological features of the data (cf. \cite{cai2020elder}). For 2-parameter persistence modules, the multigraded Betti numbers are also called the \emph{bigraded} Betti numbers. RIVET \cite{lesnick2015interactive} represents the bigraded Betti numbers of a 2-parameter persistence module as a collection of colored dots in the plane. More interestingly, RIVET employs the bigraded Betti numbers to implement an interactive visualization of the fibered barcode. Recently, Lesnick-Wright \cite{lesnick2019computing} and Kerber-Rolle \cite{kerber2021fast} developed efficient algorithms for computing minimal presentations and the bigraded Betti numbers of 2-parameter persistence modules.

\paragraph{Persistence diagram and its generalizations.} In most applications of 1-parameter persistent homology, the notion of \emph{persistence diagram} \cite{edelsbrunner2000topological,landi1997new} (or equivalently \emph{barcode} \cite{carlsson2005persistence}; cf. Definition \ref{def:barcode})  plays a central role. The persistence diagram of any $M:\Rb\rightarrow \vect$ is not only a visualizable topological summary of $M$, but also a \emph{stable} and \emph{complete} invariant of $M$ \cite{cohen2007stability}. In contrast, as mentioned before, there is no simple complete invariant for $d$-parameter persistence modules when $d>1$.

Patel introduced the notion of \emph{generalized persistence diagram} for \emph{constructible} functors $\Rb\rightarrow \Ccal$, in which $\Ccal$ satisfies certain properties \cite{patel2018generalized}. Construction of the generalized persistence diagram is based on the observation that the persistence diagram of $M:\Rb\rightarrow \vect$ 
\cite{edelsbrunner2000topological} is an instance of the M\"obius inversion of the \emph{rank invariant} \cite{carlsson2009theory} of $M$. McCleary and Patel showed that the generalized persistence diagram is stable when $\Ccal$ is a skeletally small abelian category \cite{mccleary2018bottleneck}.  Kim and M\'emoli further extended Patel's generalized persistence diagram to the setting of functors $\Pb\rightarrow \Ccal$ in which $\Pb$ is a \emph{essentially} finite poset such as a finite $d$-dimensional grid  \cite{kim2018generalized}. The generalized  persistence diagram of $\Pb\rightarrow \Ccal$ is defined as the M\"obius inversion of the \emph{generalized rank invariant} of $\Pb\rightarrow \Ccal$. The generalized  persistence diagram is not only a complete invariant 
of interval decomposable persistence modules $\Pb\rightarrow \vect$ (Theorem \ref{thm:GPD is barcode proxy}), but is also well-deﬁned regardless of the interval decomposability. The generalized rank invariant of $\Rb^d\mbox{(or $\Zb^d$)}\rightarrow \vect$ is proven to be stable with respect to a certain generalization of the erosion distance \cite{patel2018generalized} and the interleaving distance  \cite{lesnick2015theory} (see the latest version of the arXiv preprint of \cite{kim2018generalized}).

\paragraph{Our contributions.} 
Assume that a given $M:\Zb^2\rightarrow \vect$ is finitely generated. We establish a combinatorial formula for extracting the bigraded Betti numbers of $M$ from the generalized persistence diagram of $M$ (Theorem \ref{thm:pictorial formula}). More interestingly, the formula we found is a generalization of a well-known formula for extracting the bigraded Betti numbers from interval decomposable persistence modules (Theorem \ref{thm:bigraded bettis for an interval2}). 

Namely, for any finitely generated \emph{interval decomposable} $M:\Zb^2\rightarrow\vect$, there is a visually intuitive  way to find the bigraded Betti numbers of $M$ from the indecomposable summands of $M$. An example of this process is shown in Fig. \ref{fig:introduction} (A)-(C). For \emph{any} finitely generated $N:\Zb^2\rightarrow \vect$, which may not be interval decomposable, we utilize a similar process to find the bigraded Betti numbers of $N$ from the \emph{($\inter$-)generalized persistence diagram} of $N$. This process is shown in Fig. \ref{fig:introduction} (A')-(C'). In a sense, Theorem \ref{thm:pictorial formula} thus  reinforces the viewpoint that the ($\inter$-)generalized persistence diagram is a proxy for the \emph{barcode} (Definition \ref{def:barcode}) of persistence modules \cite{asashiba2019approximation,kim2018generalized}.

One implication of Theorem \ref{thm:pictorial formula} is that all invariants of 2-parameter persistence modules that are
computed by the software RIVET \cite{lesnick2015interactive} are encoded by the generalized persistence diagram. In other words, we obtain the following hierarchy of invariants
for \emph{any} finitely generated $M:\Zb^2\rightarrow\vect$, where invariant A is placed
above invariant B if invariant B can be recovered from invariant A:
\begin{figure}
    \centering
    \includegraphics[width=\textwidth]{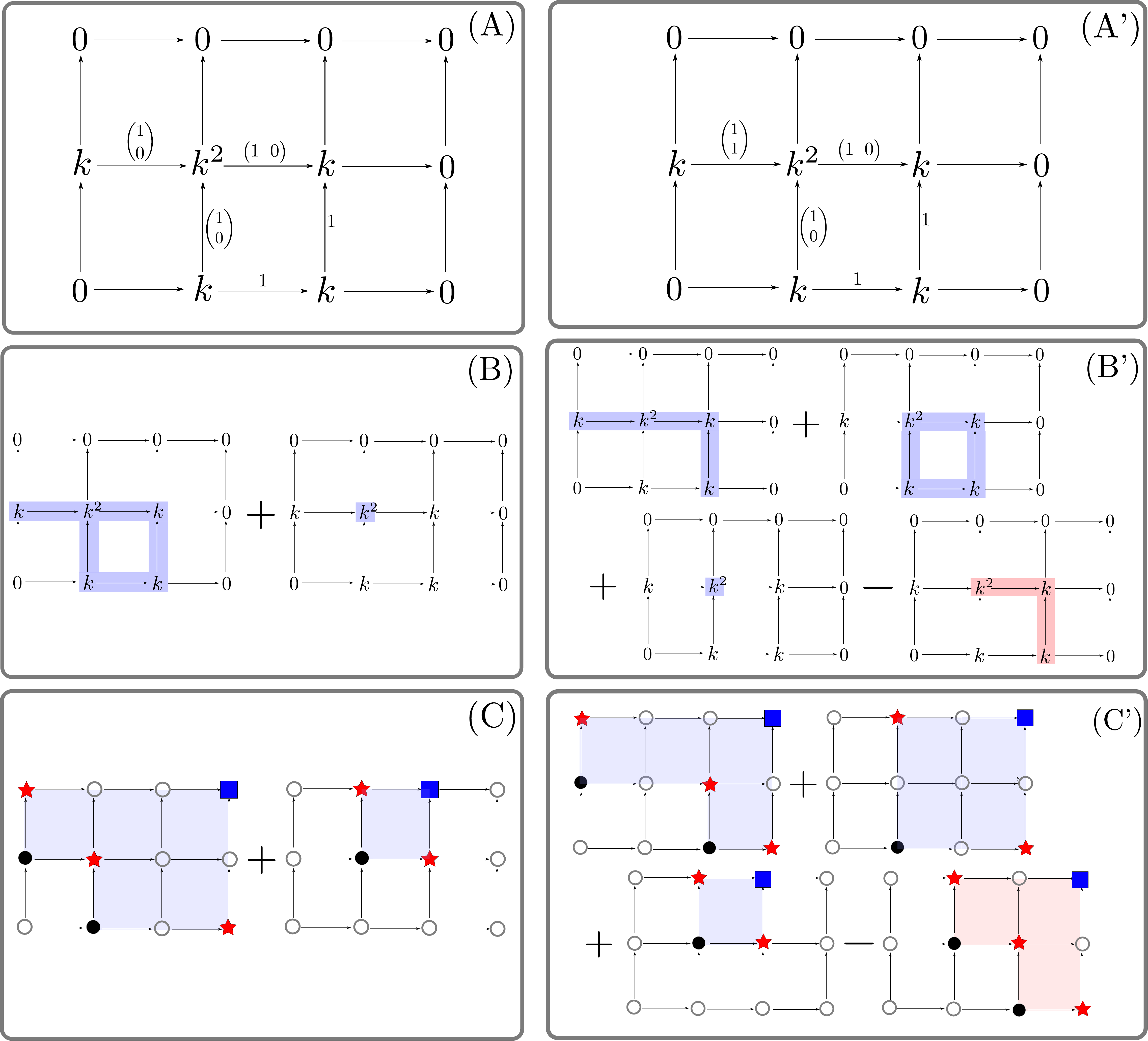}
    \caption{(A) A $\Zb^2$-indexed persistence module $M$ whose support is contained in a $3\times 4$ grid. (B) $M$ is interval decomposable, and the barcode of $M$ consists of the two blue intervals of $\Zb^2$ (Definitions \ref{def:intervals} and \ref{def:barcode}). (C) Expand each of the blue intervals from (B) to intervals in $\Rb^2$ as follows: Each point $p=(p_1,p_2)$ in the two intervals is expanded to the unit square $[p_1,p_1+1)\times[p_2,p_2+1)\subset \Rb^2$. Black dots, red stars and blue squares indicate three different \textit{corner types} of the expanded intervals (see Fig. \ref{fig:corner_points}). The bigraded Betti numbers of $M$ can be read from these corner types; for each $p\in \Zb^2$, $\beta_j(M)(p)$ is equal to the number of black dots, red stars, and blue squares at $p$ when $j=0,1,2$, respectively. (A') Another $\Zb^2$-indexed persistence module $N$ whose support is contained in a $3\times 4$ grid. $N$ is not interval decomposable. (B') The $\inter$-generalized persistence diagram of $N$ (Definition \ref{def:IGPD}) is shown, where the multiplicity of the red interval is -1 and the multiplicity of each blue interval is 1. (C') is similarly interpreted as in (C), where corner points of the red interval negatively contribute to the counting of the bigraded Betti numbers.  More details are provided in Example \ref{ex:pictorial interpretation}.}
    \label{fig:introduction}
\end{figure}

\begin{center}
\begin{tikzpicture}
    \node (top) at (0,0) {Generalized persistence diagram};
    \node (middle left) at (-3,-1)  {Fibered barcode};
    \node (middle right) at (3,-1) {Bigraded Betti numbers};
    \node (bottom) at (0,-2) {Hilbert function};
    \draw [thick, shorten <=-2pt, shorten >=-2pt] (top) -- (middle left);
    \draw [thick, shorten <=-2pt, shorten >=-2pt] (top) -- (middle right);
     \draw [thick, shorten <=-2pt, shorten >=-2pt] (middle left) -- (bottom);
     \draw [thick, shorten <=-2pt, shorten >=-2pt] (middle right) -- (bottom);
\end{tikzpicture}
\end{center}
We remark that the generalized persistence diagram is equivalent to the generalized rank invariant (Definitions \ref{def:GPD} and \ref{def:IGPD}). Also, the fibered barcode is equivalent to the (standard) rank invariant \cite{carlsson2009theory}. 
Hence, in the diagram above,  \textit{generalized persistence diagram} and \textit{fibered barcode} can be replaced by \textit{generalized rank invariant} and \textit{rank invariant}, respectively.

In the course of establishing Theorem \ref{thm:pictorial formula}, we verify that the \emph{interval decomposable approximation} \new{and \emph{multirank invariant}} of 2-parameter persistence modules
(introduced by Asashiba et al. \cite{asashiba2019approximation}\new{ and Thomas \cite{thomas2019invariants}, respectively}) also 
\new{encode} the bigraded Betti numbers (Remark \ref{rem:asashiba} and Corollary \ref{cor:multirank determines bettis}).

\begin{remark}\label{rem:meaning}
    It should not  be construed that Theorem \ref{thm:pictorial formula} provides a practically efficient way to compute the bigraded Betti numbers. Rather, we hope that the aforementioned efficient algorithms to compute the bigraded Betti numbers could be useful for approximating the generalized persistence diagram.
\end{remark}

\new{We also show that for $d\geq 3$, the generalized persistence diagram does not determine the multigraded Betti numbers of $d$-parameter persistence modules (Theorem \ref{thm:no pictorial formula}).} 
\new{Lastly, we show that, in general, the generalized persistence diagram
and the multirank invariant do not determine each other (Examples \ref{ex:identical GPD, different mrk} and \ref{ex:identical mrk, different GPD}). }

\paragraph{Other related work.} McCleary and Patel utilized the M\"obius inversion formula for establishing a functorial pipeline to summarize simplicial filtrations over finite lattices into persistence diagrams \cite{mccleary2020edit}. Botnan et al. introduced notions of signed barcode and rank decomposition for encoding the rank invariant of multiparameter persistence modules as  a linear combination of rank invariants of indicator modules \cite{botnan2021signed}. In their paper, M\"obius inversion was utilized for computing the rank decomposition, characterizing the generalized persistence diagram in terms of rank decompositions. Asashiba et al. provided a criterion for determining whether or not a given multiparameter persistence module is interval decomposable without having to explicitly compute indecomposable decompositions  
\cite{asashiba2018interval}.  
Dey and Xin proposed an efficient algorithm for decomposing
multiparameter persistence modules and introduced a notion of \emph{persistent} graded Betti numbers, a refined version of the graded Betti numbers \cite{dey2019generalized}. Dey et al. reduced the problem of computing the generalized rank invariant of a given 2-parameter persistence module to computing the indecomposable decompositions of zigzag persistence modules \cite{dey2021computing}. Blanchette et al. developed a theoretical framework for building new invariants of a  persistence module over a poset using homological algebra \cite{blanchette2021homological}.

\paragraph{Organization.} In Section \ref{sec:background}, we review the 
notions of persistence modules, multigraded Betti numbers, and 
generalized persistence diagrams. In Section \ref{sec:BettiNumbers}, we show that the bigraded Betti numbers can be recovered from the generalized persistence diagram. In Section \ref{sec:discussion}, we discuss open questions. \new{In the appendix, we prove that (a) the multirank invariant (introduced in \cite{thomas2019invariants}) also determines the bigraded Betti numbers, and that (b) in general, the generalized persistence diagram and the multirank invariant do not determine each other.}

\paragraph{Acknowledgments.}
The authors would like to thank Dr. Ezra Miller, Dr. Rich\'ard Rim\'anyi, Dr. Facundo M\'emoli, and Dr. Amit Patel for their invaluable comments. We also thank anonymous reviewers for their invaluable feedback on earlier versions of the paper. 
Samantha Moore is supported by a National Science Foundation Graduate Research Fellowship under Grant No. 1650116.

\section{Preliminaries}\label{sec:background}

In Section \ref{sec:persistence modules}, we review the notions of persistence modules and interval decomposability. In Section \ref{sec:bigraded betti numbers}, we recall the notion of multigraded Betti numbers (an invariant of multiparameter persistence modules). In Section \ref{sec:mobius inversion}, we review the  M\"obius inversion formula in combinatorics. In Section \ref{sec:generalized rank invariant}, we review the notions of generalized rank invariant and generalized persistence diagram. In Section \ref{sec:m times n module}, we provide a formula of the ($\inter$-)generalized persistence diagram in a certain setting, which will be useful in the next section.

\subsection{Persistence modules and their interval decomposability}\label{sec:persistence modules}

Let $\Pb$ be a poset. We regard $\Pb$ as the category that has points of $\Pb$ as its objects and for $p,q\in\Pb$ there is a unique morphism $\ba\rightarrow \bb$ if and only if $\ba\leq \bb$ in $\Pb$. For $d\in\N$, let $\Rb^d$ and subsets of $\Rb^d$ (such as $\Zb^d$) be given the partial order defined by $(a_1,a_2,\ldots,a_d)\leq (b_1,b_2,\ldots,b_d)$ if and only if $a_i\leq b_i$ for $i=1,2,\ldots,d$.  

Every vector space in this paper is over some fixed field $\F$. Let $\vect$ denote the category of \emph{finite dimensional} vector spaces  and linear maps over $\F$.

A $\Pb$\textbf{-indexed persistence module}, or simply a \textbf{$\Pb$-module}, refers to a functor $M:\Pb\rightarrow \vect$. In other words, to each $\ba\in \Pb$, a vector space $M(\ba)$ is associated, and to each pair $\ba\leq \bb$ in $\Pb$, a linear map $\varphi_{M}(\ba,\bb):M(\ba)\rightarrow M(\bb)$ is associated. Importantly, whenever $p\leq q\leq r$ in $\Pb$, it is required that $\varphi_M(p,r)=\varphi_M(q,r)\circ \varphi_M(p,q)$. When $\Pb=\Rb^d$ or $\Zb^d$, $M$ is also called a \textbf{$d$-parameter persistence module}. 

Consider a \textit{zigzag poset} of $n$ points, 
\begin{equation}\label{eq:zigzag poset}
    \bullet_{1}\leftrightarrow \bullet_{2} \leftrightarrow \ldots \bullet_{n-1} \leftrightarrow \bullet_n
\end{equation}
where $\leftrightarrow$ stands for either $\leq$ or $\geq$. A functor from a zigzag poset (of $n$ points) to $\vect$ is called a \textbf{zigzag module} (of length $n$) \cite{carlsson2010zigzag}. 

A \textbf{morphism} between  $\Pb$-modules $M$ and $N$ is a natural transformation $f:M\rightarrow N$ between $M$ and $N$. That is, $f$ is a collection $\{f_{\ba}: M(\ba)\rightarrow N(\ba)\}_{\ba\in \Pb}$ of linear maps such that for every pair $\ba\leq\bb$ in $\Pb$, the following diagram commutes:
 \[\begin{tikzcd}
M(\ba) \arrow{r}{\varphi_M(\ba,\bb)} \arrow{d}{f_{\ba}}
&M(\bb) \arrow{d}{f_{\bb}}\\
N(\ba) \arrow{r}{\varphi_N(\ba,\bb)} &N(\bb).
\end{tikzcd}\]

The \textbf{kernel} of $f$, denoted by $\ker(f):\Pb\rightarrow \vect$, is defined as follows: For $p\in \Pb$, $\ker(f)(p):=\ker(f_p)\subseteq M(p)$. For $p\leq q$ in $\Pb$, $\varphi_{\ker(f)}(p,q)$ is the restriction of $\varphi_M(p,q)$ to $\ker(f_p)$. 
Two $\Pb$-modules $M$ and $N$ are (naturally) \textbf{isomorphic}, denoted by $M\cong N$, if there exists a natural transformation $\{f_{\ba}\}_{\ba\in \Pb}$ from $M$ to $N$ where each $f_{\ba}$ is an isomorphism. 

The \textbf{direct sum} $M\bigoplus N$ of $M,N:\Pb\rightarrow \vect$ is the $\Pb$-module where $(M\bigoplus N)(p)=M(p)\bigoplus N(p)$ for $p\in \Pb$ and $\varphi_{M\bigoplus N}(p,q)=\varphi_{M}(p,q)\bigoplus \varphi_{N}(p,q)$ for $p\leq q$ in $\Pb$.  A nonzero $\Pb$-module $M$ is \textbf{indecomposable} if whenever $M=M_1\bigoplus M_2$ for some $\Pb$-modules $M_1$ and $M_2$, either $M_1=0$ or $M_2=0$.

\begin{theorem}[Krull-Remak-Schmidt-Azumaya \cite{azumaya1950corrections}]
\label{thm:direct sum decomposable} 
Any $\Pb$-module $M$ has a direct sum decomposition $M\cong \bigoplus\limits_{i} M_i$ where each $M_i$ is indecomposable. Such a decomposition is unique up to isomorphism and reordering of the summands.
\end{theorem}
In what follows, we review the notion of interval decomposability. 

\begin{definition}\label{def:intervals}
Let $\Pb$ be a poset. An \textbf{interval} of $\Pb$ is a subset $I\subseteq \Pb$ such that: 	
\begin{enumerate*}[label=(\roman*)]
    \item $I$ is nonempty.    
    \item If $\ba,\bb\in I$ and $\ba\leq \bc\leq \bb$, then $\bc\in I$. \label{item:convexity}
    \item $I$ is \textbf{connected}, i.e. for any $\ba,\bb\in I$, there is a sequence $\ba=\ba_0,
		\ba_1,\cdots,\ba_\ell=\bb$ of elements of $I$ with either $\ba_i\leq\ba_{i+1}$ or $\ba_{i+1}\leq \ba_i$ for each $i\in[0,\ell-1]$.\label{item:interval3}\footnote{This definition of interval is not the standard definition of interval used in order theory but it is often used in the literature concerned with  persistence modules over posets; e.g. \cite{botnan2020decomposition}. 
In order theory language, $I$ is a nonempty convex connected subset of $\Pb$.} 
\end{enumerate*}
By $\inter(\Pb)$, we denote the set of all intervals of $\Pb$.
\end{definition}

For example, any interval of a zigzag poset in (\ref{eq:zigzag poset}) is a set of consecutive points in $\{\bullet_1,\bullet_2,\ldots,\bullet_n\}$. 

For an interval $I$ of a poset $\Pb$, the \textbf{interval module} $V_I:\Pb\rightarrow \vect$ is defined as
\[V_I(\ba)=\begin{cases}
k&\mbox{if}\ \ba\in I\\0
&\mbox{otherwise,} 
\end{cases}\hspace{20mm} \varphi_{V_I}(\ba,\bb)=\begin{cases} \mathrm{id}_k& \mbox{if} \,\,\ba,\bb\in I,\ \ba\leq \bb\\ 0&\mbox{otherwise.}\end{cases}\]

It is well-known that any interval module is indecomposable \cite[Proposition 2.1]{botnan2018algebraic}.

\begin{definition}\label{def:barcode}
A $\Pb$-module $M$ is said to be \textbf{interval decomposable} if there exists a multiset $\barc(M)$ of intervals  of $\Pb$ such that $M\cong \bigoplus\limits_{\mathclap{I\in\barc(M)}}V_I.$
We call $\barc(M)$ the \textbf{barcode} of $M.$
\end{definition}

\begin{theorem}[\cite{azumaya1950corrections,crawley2015decomposition,gabrielsthm}]\label{thm:interval decomposable}
For $d=1$, any  $M:\Rb^d\ \mbox{(or $\Zb^d$)}\ \rightarrow\vect$ is interval decomposable and thus admits a (unique) barcode. However, for $d\geq 2$, $M$ may not be interval decomposable. Lastly, any zigzag module is interval decomposable and thus admits a (unique) barcode. 
\end{theorem}

The following notation is useful in the rest of the paper.

\begin{notation}\label{def:multiplicity} 
Assume that a $\Pb$-module $M$ is isomorphic to the direct sum $\bigoplus_{i\in \Ical} M_i$ for some indexing set $\Ical$ where each $M_i$ is indecomposable. For $I\in \inter(\Pb)$, we define $\mult(I,M)$ as the cardinality of the set $\{i\in \Ical: M_i\cong V_I\}.$ In words, $\mult(I,M)$ is the number of those summands $M_i$ which are isomorphic to the interval module $V_I$.
\end{notation}

\subsection{Multigraded Betti numbers}\label{sec:bigraded betti numbers}

In this section we review the notion of multigraded Betti numbers \cite{eisenbudbetti}.

Fix any $p\in\Zb^d$. Then, the \emph{upper} set $p^\uparrow:=\{x\in \Zb^d: p\leq x\}$ determines an interval of $\Zb^d$. An $\Zb^d$-module $F$ is \textbf{free} if there exists $p_1, p_2, \cdots, p_n$ in $\Zb^d$ such that $F\cong\bigoplus\limits_{i=1}^n V_{{p_i}^{\uparrow}}$. 

Let $M$ be an $\Zb^d$-module. An element $v \in M(p)$ for some $p\in \Zb^d$ is called a \textbf{homogeneous} element of $M$. Assume that $M$ is \textbf{finitely generated}, i.e. there exist $p_1,\ldots,p_n \in \Zb^d$ and $v_i\in M(p_i)$ for $i=1,\ldots, n$ such that for any $p\in \Zb^d$ and for any nonzero $v\in M(p)$, there exist $c_i\in \F$ for $i=1,\ldots,n$ with \[v=\sum\limits_{i=1}^n c_i\cdot \varphi_M(p_i,p) (v_i).\] The collection $\{v_1,\ldots, v_n\}$ is called a (homogeneous) \textbf{generating set} for $M$.

Let us assume that $\{v_1, \ldots, v_n\}$ is a \emph{minimal} homogeneous generating set for $M$, i.e. there is no homogeneous generating set for $M$ that includes fewer than $n$ elements. 
Let $F_0:=\bigoplus\limits_{i=1}^n V_{p_i^{\uparrow}}$. 
For $i=1,\ldots,n$, let $1_{p_i}\in V_{p_i^{\uparrow}}(p_i)$.
Then, the set $\{1_{p_1},\ldots,1_{p_n}\}$ generates $F_0$ and 
the morphism $\eta_0 : F_0\rightarrow M$ defined by $\eta_0 (1_{p_i})=v_i$ for $i=1,\ldots,n$ is surjective. Let $K_0:=\ker (\eta_0)\subseteq F_0$ and let $\imath_0:K_0\hookrightarrow F_0$ be the inclusion map. Iterate this process using $K_0$ in place of $M$.\footnote{
We remark that $K_0$ is finitely generated by the following two facts. (1) A submodule of any finitely generated module over a Noetherian ring is finitely generated. (2) A finitely generated $\Zb^d$-module can be viewed as a module over the polynomial ring 
in $d$ variables, 
which is Noetherian; this viewpoint can be found in \cite{carlsson2009theory,miller2005combinatorial} for example.} 

Namely, identify a minimal homogeneous generating set $\{v_1',\ldots,v_m'\}$ for $K_0$ where $v_j'\in (K_0)_{p_j'}$ for some $p_1',\ldots,p_m'\in \Zb^d$ and consider the free module $F_1:=\bigoplus\limits_{j=1}^m V_{p_{j'}^{\uparrow}}$ and the surjection $\eta_1:F_1\rightarrow K_0$. Then we have the map $\imath_0\circ \eta_1:F_1\rightarrow F_0$. By repeating this process, we obtain a \textbf{minimal free resolution of $M$:} \[\cdots\xrightarrow{}F_2\xrightarrow{\imath_1 \circ\eta_2} F_1\xrightarrow{\imath_0\circ\eta_1} F_0\xrightarrow{\eta_0} M \xrightarrow{} 0.\] 
This resolution is unique up to isomorphism \cite[Theorem 1.6]{eisenbudbetti}. Hilbert's Syzygy Theorem  guarantees that $F_j=0$ for $j> d$ \cite{hilbertsyzygy}.

\begin{definition}
For $j=0,\ldots,d$, the $j^{\mathrm{th}}$ \textbf{multigraded Betti number} $\beta_j(M):\Zb^d\rightarrow \vect$ of $M$ is defined by mapping each $p\in \Zb^d$ to
\begin{equation*}\beta_j (M)(p):= \mult\left(p^{\uparrow}, F_j\right) \ \mbox{(Notation \ref{def:multiplicity})}.
\end{equation*}
When $d=2$, the multigraded Betti numbers are also called the \textbf{bigraded Betti numbers}. 
\end{definition}

We will see that for an interval decomposable $\Zb^2$-module $M$, its bigraded Betti numbers can be extracted from $\barc(M)$. To this end, we will make use of a certain regions that arise by ``blowing-up" intervals from $\barc(M)$:

\begin{definition}\label{def:corresponding region} Given any $I\in\inter(\Zb^2)$, the subset of $\Rb^2$ 
\begin{equation}\label{eq:region}
    I^+:=\bigcup_{(p_1,p_2)\in I}[p_1,p_1+1)\times [p_2,p_2+1)
\end{equation}
 will be referred to as the region corresponding to $I$ in $\Rb^2$.
\end{definition}

The following remarks are well-known; e.g.  \cite[Remarks 2.4 and 3.10]{cai2020elder}.

\begin{remark}\label{rem:bigraded bettis for an interval} 
\begin{enumerate}[label=(\roman*)]
    \item For any finitely generated $M,N:\Zb^d\rightarrow \vect$, we have $\beta_j(M\bigoplus N)=\beta_j(M)+\beta_j(N)$ for  $j=0,\ldots,d$.  \label{item:bigraded betti is additive}
    \item Let $I\in  \inter(\Zb^2)$. For the interval module $V_I:\Zb^2\rightarrow \vect$, the $j^{\mathrm{th}}$ bigraded Betti number $\beta_j(V_I)(p)$ is equal to $1$ if $p$ is a $j^{\mathrm{th}}$ type corner point of $I^{+}$ and is equal to $0$ otherwise; see Fig. \ref{fig:corner_points}.\label{item:bigraded bettis for an interval1}
\end{enumerate}
\end{remark}
\begin{figure}
    \centering
    \includegraphics[width=0.6\textwidth]{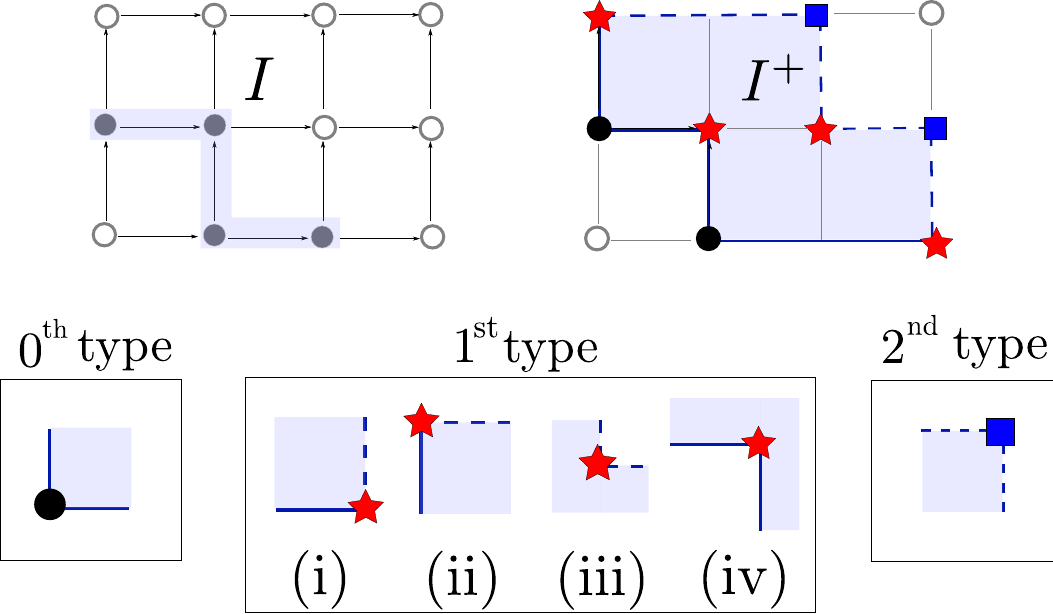}
    \caption{
    An interval $I\in\inter(\Zb^2)$ and  its corresponding region $I^+\subset \Rb^2$ with its corner points.
    Points on the upper boundary (dashed lines) do not belong to $I^+$, while points on the lower boundary (solid lines) belong to $I^+$. Points which lie on both boundaries do not belong to $I^+$.}
    \label{fig:corner_points}
\end{figure}
Remark \ref{rem:bigraded bettis for an interval} directly implies:

\begin{theorem}\label{thm:bigraded bettis for an interval2}
  Given any finitely generated \emph{interval decomposable} module $M:\Zb^2\rightarrow \vect$, the bigraded Betti numbers of $M$ can be extracted from $\barc(M)$. More specifically, the bigraded Betti numbers of $M$ can be extracted from the corner points of the elements in the multiset $$\{I^+\subset \Rb^2:I\in \barc(M)\}.$$ 
\end{theorem}

In Theorem \ref{thm:pictorial formula}, we remove the assumption that $M$ be interval decomposable and generalize Theorem \ref{thm:bigraded bettis for an interval2}  to the setting of \emph{any} finitely generated $\Zb^2$-modules.

\subsection{The M\"obius inversion formula in combinatorics}\label{sec:mobius inversion}

In this section, we briefly review the M\"obius inversion formula, a fundamental concept in combinatorics  \cite{bender1975applications,rota1964foundations}.

A poset $\QA$ is said to be \textbf{locally finite} if for all $p,q\in \QA$ with $p\leq q$, the set $[p,q]:=\{r\in \QA: p\leq r\leq q\}$ is finite. Let $\QA$ be a locally finite poset. The M\"obius function $\mu_{\QA}:\QA\times\QA\rightarrow \Zb$ of $\QA$ is defined\footnote{More precisely, the codomain of $\mu_{\QA}$ is the multiple of $1$ in a specified base ring.}
recursively as

\begin{equation}\label{eq:induction}
    \mu_{\QA}(p,q)=\begin{cases}1,& \mbox{$p=q$,}\\ -\sum\limits_{p\leq r< q}\mu_{\QA}(p,r),& \mbox{$p<q$,} \\ 0,& \mbox{otherwise.} \end{cases}
\end{equation}

For $q_0\in \QA$, consider the principal ideal $q_0^{\downarrow}:=\{q\in\QA: q\leq q_0\}$. Note that if we assume that $q^{\downarrow}$ is finite for \emph{all} $q\in \QA$, then $\QA$ must be locally finite. To see this, note that, for any $p,q\in \QA$ with $p\leq q$, the set $[p,q]$ is a subset of the finite set $q^{\downarrow}$.

\begin{theorem}[M\"obius Inversion formula]\label{thm:mobius}
Assume that $q^{\downarrow}$ is finite for all $q\in \QA$. Let $k$ be a field. For any pair of functions $f,g:\QA\rightarrow \F$,   \[g(q)=\sum_{r\leq q} f(r)\ \mbox{for all $q\in \QA$}\]  if and only if 
\[f(q)=\sum_{r\leq q} g(r)\cdot \mu_{\QA}(r,q) \ \mbox{for all } q\in \QA.\]
\end{theorem}
The function $f$ is called the \emph{M\"obius inversion} of $g$.

One interpretation of the M\"obius inversion formula is that of a discrete analogue of the derivative of a real-valued map in elementary calculus, as explained in the following example:
\begin{example}\label{ex:mobius inversion} 
Let $[m]=\{0,1,\ldots,m\}$ with the usual order. Then, \[\mu_{[m]}(a,b)=\begin{cases}1,& \mbox{$a=b$,}\\ -1,& \mbox{$a=b-1$,} \\ 0,& \mbox{otherwise.} \end{cases}\]
Hence, for any function $g:[m]\rightarrow \Rb$, its M\"obius inversion $f:[m]\rightarrow \Rb$ is given by $f(a)=g(a)-g(a-1)$ for $a\neq 0$ and $f(0)=g(0)$. Hence,  at each point $a\neq 0$, $f(a)$ captures the rate of change of $g$ around that point.
\end{example}

\subsection{Generalized rank invariant and generalized persistence diagrams}\label{sec:generalized rank invariant}
In this section we review the notions of generalized rank invariant and generalized persistence diagram \cite{kim2018generalized,patel2018generalized}.
\begin{framed}
    Throughout this subsection, let $\Pb$ denote a \emph{finite connected} poset  (Definition \ref{def:intervals} \ref{item:interval3}).
\end{framed}
 Consider any $\Pb$-module $M$. Then $M$ admits a limit and a colimit of $M$: $\varprojlim M=(L, (\pi_p:L\rightarrow M(p))_{p\in \Pb})$ and $\varinjlim M =(C, (\iota_p:M(p)\rightarrow C)_{p\in \Pb})$; 
see the appendix for a  review of the definitions of limits and colimits (Definitions \ref{def:limit} and \ref{def:colimit}). 
This implies that, for every $p\leq q$ in $\Pb$, \[M(p\leq q)\circ \pi_p =\pi_q\ \ \mbox{and }\  \iota_q\circ M(p\leq q) =\iota_p.\ \]  Since $\Pb$ is connected, these equalities imply that $\iota_p \circ \pi_p=\iota_q \circ \pi_q:L\rightarrow C$ for any $p,q\in \Pb$. In words, the composition $\iota_p \circ \pi_p$ is independent of $p$. The \textbf{canonical limit-to-colimit map $\psi_M:\varprojlim M\rightarrow \varinjlim M$} is therefore defined to be \emph{the} linear map $\iota_p\circ \pi_p$ where $p$ is \emph{any} point in $\Pb$.

\begin{definition}[\cite{kim2018generalized}]\label{def:generalized rank} The \textbf{rank} of $M:\Pb\rightarrow \vect$ is defined as the rank of the canonical limit-to-colimit map $\psi_M:\varprojlim M \rightarrow \varinjlim M$.
\end{definition}

The rank of $M:\Pb\rightarrow \vect$ counts the multiplicity of the fully supported interval module $V_{\Pb}$ in a direct sum decomposition of $M$ into indecomposable modules:

\begin{theorem}[{\cite[Lemma 3.1]{chambers2018persistent}}]
\label{thm:rkequalsintervals}
For any $M:\Pb\rightarrow \vect$, the rank of $M$ is equal to $\mult(\Pb,M)$.
\end{theorem}

Let $p,q\in \Pb$. We say that $p$ \emph{covers} $q$ and write $q\triangleleft p$ if $q< p$ and there is no $r\in \Pb$ such that $q<r<p$.

A  subposet $I\subseteq \Pb$  is said to be \textbf{path-connected in $\Pb$} if for any $p\neq q$ in $I$, there exists a sequence $p=p_0,p_1,\ldots,p_n=q$ in $I$ such that either $p_i\triangleleft p_{i+1}$ or $p_{i+1}\triangleleft p_i$ in $\Pb$ for $i=0,\ldots,n-1$.
For example, the set $\{0,2\}$ is a \emph{connected} (Definition \ref{def:intervals} \ref{item:interval3}) subposet of $\{0,1,2\}$ equipped with the usual order, but is not path-connected in $\{0,1,2\}$.

By $\con(\Pb)$ we denote the poset of all path-connected subposets of $\Pb$ that is ordered by inclusions. We remark that, since $\Pb$ is finite, $\con(\Pb)$ is finite. For example, assume that $\Pb$ is the zigzag poset $\{\bullet_1< \bullet_2 > \bullet_3\}$. Then, $\con(\Pb)$ consists of the six elements: $\{\bullet_1\},\{\bullet_2\},\{\bullet_3\},\{\bullet_1,\bullet_2\},\{\bullet_2,\bullet_3\}$,and $\{\bullet_1,\bullet_2,\bullet_3\}$. All of these are also intervals of  $\{\bullet_1< \bullet_2 > \bullet_3\}$, i.e. $\inter(\Pb)=\con(\Pb)$ (Definition \ref{def:intervals}). In general, $\inter(\Pb)$ is a subposet of $\con(\Pb)$.

\begin{definition}\label{def:generalized rk inv} 
The \textbf{generalized rank invariant} of $M:\Pb\rightarrow \vect$ is the function \[\rk(M):\con (\Pb)\rightarrow \Zplus\]
which maps $I\in\con(\Pb)$ to the rank of the restriction $M\vert{}_I$ of $M$.
\end{definition}

In fact, in order to define the generalized rank invariant, $\Pb$ does not need to be finite \cite[Section 3]{kim2018generalized}. However, for this work, it suffices to consider the case when $\Pb$ is finite. 

\begin{remark}\label{rem:rank invariant properties} 
Let $I\in \con(\Pb)$. For any $\Pb$-module $M$, the following hold:
\begin{enumerate}[label=(\roman*)] 
    \item By Theorem \ref{thm:rkequalsintervals}, if there exists $p\in I$ such that $M(p)=0$, then $\rk(M)(I)=0$.\label{item:rank invariant properties2}
    \item Let $I,J\in \con(\Pb)$ with $J\supseteq I$. Then $\rk(M)(J)\leq \rk(M)(I)$, i.e. $\rk(M)$ is order-reversing. This is because the canonical limit-to-colimit map $\varprojlim M\vert{}_{I}\rightarrow \varinjlim M\vert{}_{I}$ is a factor of the canonical limit-to-colimit map $\varprojlim M\vert{}_{J}\rightarrow \varinjlim M\vert{}_{J}$ \cite[Proposition 3.7]{kim2018generalized}.  This monotonicity implies that if $\rk(M)(I)=0$, then  $\rk(M)(J)=0$.\label{item:rank invariant properties1}
\end{enumerate}
\end{remark}

The following is a  corollary of Theorem \ref{thm:rkequalsintervals}.

\begin{proposition}[{\cite[Proposition 3.17]{kim2018generalized}}]\label{prop:rk}
Let $M:\Pb\rightarrow \vect$ be interval decomposable. Then for any $I\in \con(\Pb)$,  
\[\rk(M)(I)=\sum_{\substack{J\supseteq I\\ J\in\inter(\Pb)}} \mult(J,M).\]
In words, $\rk(M)(I)$ equals the total multiplicity of intervals $J$ in $\barc(M)$ that contain $I$.
\end{proposition}

For any poset $\QA$, let $\QA^\mathrm{op}$ denote the \textbf{opposite poset} of $\QA$, i.e. $p\leq q$ in $\QA$ if and only if $q\leq p$ in $\QA^\mathrm{op}$. By virtue of Theorem \ref{thm:mobius} we have:

\begin{definition}\label{def:GPD} Let $\Pb$ be a finite connected poset. The \textbf{generalized persistence diagram} of $M:\Pb\rightarrow \vect$ is the unique function $\dgm(M):\con(\Pb)\rightarrow \Zb$ that satisfies, for any $I\in \con(\Pb)$, 
\begin{equation}\label{eq:GPD0}
    \rk(M)(I)=\sum_{\substack{J\supseteq I\\ J\in\con(\Pb)}}\dgm(M)(J).
\end{equation}
In other words, $\dgm(M)$ is the M\"obius inversion of $\rk (M)$ over $\conop (\Pb)$. That is, for $I\in\con(\Pb)$, 
\begin{equation}\label{eq:GPD}
    \dgm(M)(I):= \sum\limits_{\substack{J\supseteq I\\ J\in\con(\Pb)}} \mu_{\conop(\Pb)}(J,I)\cdot \rk(M)(J).
\end{equation}
\end{definition}
The function $\mu_{\conop(\Pb)}$ has been precisely computed in \cite[Section 3]{kim2018generalized}. 

Next, we restrict the domain of $\rk(M)$ and $\dgm(M)$ to the collection $\inter(\Pb)$ of all intervals of $\Pb$. For $M:\Pb \rightarrow \vect$, let $\intrk(M)$ denote the restriction of $\rk(M):\con(\Pb)\rightarrow \Zplus$ to $\inter(\Pb)$. We consider the M\"obius inversion of $\intrk(M)$ over the poset $\inter^{\mathrm{op}} (\Pb)$. Again by virtue of Theorem \ref{thm:mobius} we have:

\begin{definition}\label{def:IGPD} 
Let $\Pb$ be a finite connected poset. The \textbf{$\inter$-generalized persistence diagram of $M:\Pb\rightarrow \vect$} is the unique function $\intdgm(M):\inter(\Pb)\rightarrow \Zb$ that satisfies, for any $I\in \inter(\Pb)$, \[\intrk(M)(I)=\sum_{\substack{J\supseteq I\\ J\in\inter(\Pb)}}\intdgm(M)(J).\]
In other words, by Theorem \ref{thm:mobius}, $\intdgm(M)$ is the M\"obius inversion of $\intrk (M)$ over $\inter^\mathrm{op} (\Pb)$, i.e. for $I\in\inter(\Pb)$, \begin{equation}\label{eq:IGPD}\intdgm(M)(I):= \sum\limits_{\substack{J\supseteq I\\ J\in\inter(\Pb)}} \mu_{\inter^\mathrm{op}(\Pb)}(J,I)\cdot \intrk(M)(J).
\end{equation}
\end{definition}

Recall from Example \ref{ex:mobius inversion} that, for $m\in \Zplus$, $[m]$ is defined as the set $\{0<1< \cdots< m\}$.

\begin{remark}\label{rem:asashiba}
In Definition \ref{def:IGPD}, let $\Pb$ be the finite product poset $[m]\times [n]$ for any $m,n\in \N\cup\{0\}$. Then, $\intdgm(M)$ is equivalent to the \emph{interval decomposable approximation} $\delta^{\text{tot}}(M)$ given in \cite{asashiba2019approximation}; this is a direct corollary of Theorem \ref{thm:rkequalsintervals}. The M\"obius function  $\mu_{\interop([m]\times[n])}$ has been precisely computed in \cite{asashiba2019approximation}, which leads to Theorem \ref{thm:mobius inversion by asashiba} below. 
\end{remark}

Although we do not require $M:\Pb\rightarrow\vect$ to be interval decomposable in order to define $\dgm(M)$ or $\intdgm(M)$, these two diagrams generalize the notion of barcode (Definition \ref{def:barcode}): 

\begin{theorem}\label{thm:GPD is barcode proxy}
Let $M:\Pb\rightarrow\vect$ be interval decomposable. Then we have:
\begin{align}
    \dgm(M)(I)&=\begin{cases}\mult(I, M)& I\in\inter(\Pb)\\0 &I\in\con(\Pb)\setminus \inter(\Pb),\ \ \mbox{and} \end{cases} \label{eq:dgm and intdgm0}\\ \intdgm(M)(I)&=\mult(I, M)\mbox{ for all $I\in \inter(\Pb)$.} \label{eq:dgm and intdgm}
\end{align}
\end{theorem}
The equality given in Equation (\ref{eq:dgm and intdgm0}) was first proved in  \cite[Theorem 3.14]{kim2018generalized}, but we include a proof here for completeness.
\begin{proof}
By Proposition \ref{prop:rk}, we have that $\displaystyle \rk(M)(I)=\sum_{\substack{J\supseteq I\\ J\in\inter(\Pb)}} \mult(I,M)$. By the uniqueness of $\dgm(M)$ in Definition \ref{def:IGPD}, $\dgm(M)(I)= \mult(I,M)$ for all $I\in \inter(\Pb)$ and  $\dgm(M)(I)=0$ for $I\in \con(\Pb)\setminus \inter(\Pb)$. By a similar argument, we have that $\intdgm(M)(I)=\mult(I, M)\mbox{ for all $I\in \inter(\Pb)$.}$  
\end{proof}

In the restricted case when $\Pb=[m]\times [n]$, the equality in Equation (\ref{eq:dgm and intdgm}) has been also independently proved in  \cite[Theorem 5.10]{asashiba2019approximation}. 

Theorem \ref{thm:GPD is barcode proxy} implies that both $\dgm(M)$ and $\intdgm(M)$ are able to completely determine the isomorphism type of an interval decomposable persistence module $M$ (which also implies that each of $\rk(M)$ and $\intrk(M)$ is strong enough to determine the isomorphism type of $M$). However, in general, the generalized persistence diagram $\dgm(M)$ is more discriminative than the $\inter$-generalized persistence diagram $\intdgm(M)$; see Example \ref{ex:GPD_vs_IGPD} in the appendix.

\medskip
In Section \ref{sec:BettiNumbers}, the case when $\Pb$ is a zigzag poset of length $3$ will be useful.

\begin{example}[{\cite[Section 3.2,2]{kim2018generalized}}]\label{ex:zigzag mobius inversion} Assume that $\Pb$ is any zigzag poset of length 3, i.e. $\bullet_1\leftrightarrow \bullet_2 \leftrightarrow \bullet_3$ where $\leftrightarrow$ stands for either $\leq$ or $\geq$.
Then, $\dgm(M)$ is computed as follows:
\begin{align*}
    \dgm(M)(\{\bullet_1\})&=\rk(M)(\{\bullet_1\})-\rk(M)(\{\bullet_1,\bullet_2\}),
\\
    \dgm(M)(\{\bullet_2\})&=\rk(M)(\{\bullet_2\})-\rk(M)(\{\bullet_1,\bullet_2\})-\rk(M)(\{\bullet_2,\bullet_3\})+\rk(M)(\{\bullet_1,\bullet_2,\bullet_3\}),
    \\
    \dgm(M)(\{\bullet_3\})&=\rk(M)(\{\bullet_3\})-\rk(M)(\{\bullet_2,\bullet_3\}),
\\
    \dgm(M)(\{\bullet_1,\bullet_2\})&=\rk(M)(\{\bullet_1,\bullet_2\})-\rk(M)(\{\bullet_1,\bullet_2,\bullet_3\}),
\\
    \dgm(M)(\{\bullet_2,\bullet_3\})&=\rk(M)(\{\bullet_2,\bullet_3\})-\rk(M)(\{\bullet_1,\bullet_2,\bullet_3\}),
\\
     \dgm(M)(\{\bullet_1,\bullet_2,\bullet_3\})&=\rk(M)(\{\bullet_1,\bullet_2,\bullet_3\}).
\end{align*}

Since $M$ is a zigzag module, it is interval decomposable (Theorem \ref{thm:interval decomposable}). Thus, we have $\dgm(M)(I)=\mult(I,M)$ for $I\in \con(\Pb)$,  the multiplicity of $I$ in $\barc(M)$. Since $\con(\Pb)=\inter(\Pb)$, each $\dgm(M)$ above can be replaced by $\intdgm(M)$.
\end{example}

\subsection{$\inter$-Generalized persistence diagram of an $([m]\times[n])$-module.}\label{sec:m times n module}

 In this section we review
 a formula of the $\inter$-generalized persistence diagram of an $([m]\times[n])$-module for any fixed integers $m,n\geq 0$. 
 
 Let us consider the poset  $\inter([m]\times [n])$. Then, given any two distinct $I,J\in \inter([m]\times [n])$, we say that $J$ \emph{covers} $I$ if $J\supsetneq I$ and there is no interval $K$ such that $J\supsetneq K\supsetneq I$.  For $I\in \inter([m]\times [n])$, let us define $\cov(I)$ as the collection of all $J\in \inter([m]\times [n])$ that cover $I$. Given any nonempty $S\subseteq \inter([m]\times [n])$,  by $\bigvee S$, we denote the smallest interval $J$ that contains all $I\in S$.

The following theorem is established by invoking Remark \ref{rem:asashiba} and finding an explicit formula for the M\"obius function $\mu_{\inter^\mathrm{op}(\Pb)}$  that appears in Equation (\ref{eq:IGPD}) with $\Pb=[m]\times [n]$.

\begin{theorem}[{\cite[Theorem 5.3]{asashiba2019approximation}}]\label{thm:mobius inversion by asashiba} For any $([m]\times [n])$-module $M$,
\begin{equation}\label{eq:asashiba}
    \intdgm(M)(I)=\intrk(M)(I)+\sum_{\substack{S\subseteq \cov(I)\\ S\neq \emptyset}}(-1)^{\abs{S}}\intrk(M)\left(\bigvee S\right).
\end{equation}
\end{theorem}

\begin{figure}
    \centering
    \includegraphics[width=\textwidth]{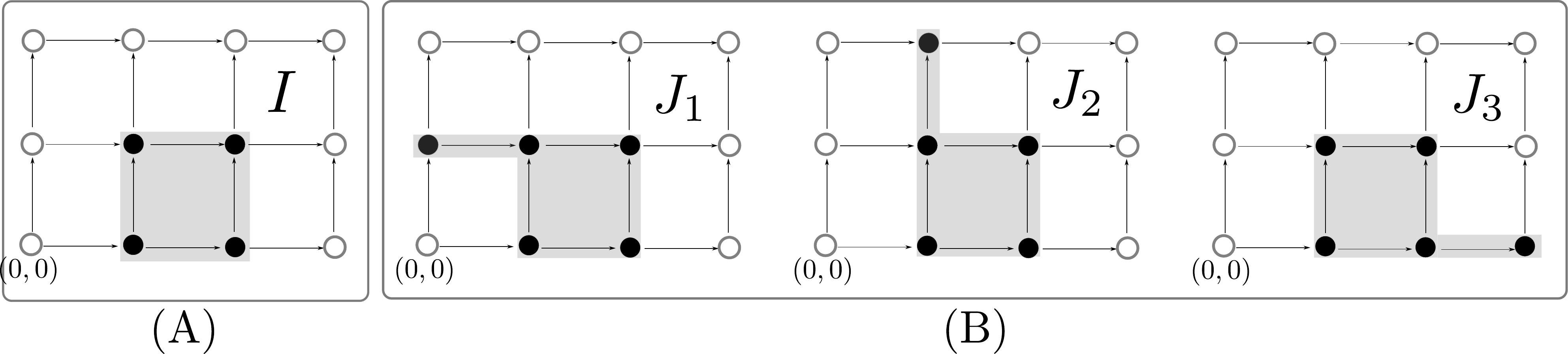}
    \caption{Illustrations for Example \ref{ex:dgm computation}.}
    \label{fig:cover}
\end{figure}

\begin{example}\label{ex:dgm computation}Let $I\in  \inter([3]\times[2])$ depicted as in Fig.~\ref{fig:cover} (A). Note that $\cov(I)=\{J_1,J_2,J_3\}$ where $J_1,J_2$ and $J_3$ are depicted as in Fig.~\ref{fig:cover} (B). For any $([3]\times[2])$-module $M$, we have:
\begin{equation*}
   \begin{aligned}
    \intdgm(M)(I)&=\intrk(M)(I)-\sum_{i=1}^3 \intrk(M)(J_i)+\sum_{i\neq j}\intrk(M)\left(\bigvee\{J_i,J_j\}\right)\\
    &-\intrk(M)\left(\bigvee\{J_1,J_2,J_3\}\right),
\end{aligned} 
\end{equation*}
where $
\bigvee\{J_1,J_2\}=J_1\cup J_2\cup\{(0,2)\}, \ \ \bigvee\{J_1,J_3\}=J_1\cup J_3, \ \ \bigvee\{J_2,J_3\}=J_2\cup J_3$,  and $\bigvee\{J_1,J_2,J_3\}=J_1\cup J_2\cup J_3\cup\{(0,2)\}$.
\end{example}

The following remark will be useful in the next section.

\begin{remark}\label{rem:rank invariant properties2} 
Let $M$ be an $([m]\times [n])$-module and let $I\in \inter([m]\times [n])$. By Remark \ref{rem:rank invariant properties} \ref{item:rank invariant properties1} and Equation (\ref{eq:asashiba}), if $\intrk(M)(I)=0$, then $\intdgm(M)(I)=0$.
\end{remark}

\section{Extracting the bigraded Betti numbers from the generalized persistence diagram}\label{sec:BettiNumbers}

In this section we aim at establishing Theorem \ref{thm:pictorial formula}, as a generalization of Theorem \ref{thm:bigraded bettis for an interval2}. 

Let $M$ be a finitely generated $\Zb^2$-module.\footnote{Main results in this section (which are Proposition \ref{prop:GPD determines Betti} and Theorem \ref{thm:pictorial formula}) also hold for finitely presented $\Rb^2$-modules upto rescaling parameters \cite{carlsson2009theory,lesnick2015interactive}.} We may assume that $M(p)=0$ for $p\not\geq (0,0)$. Then, all algebraic information of $M$ can be recovered from the restricted module $M':=M\vert{}_{[m]\times [n]}$ for some large enough positive integers $m$ and $n$. We will show that the generalized persistence diagram of $M'$ determines the bigraded Betti numbers of $M$.

\begin{definition} A given $\Zb^2$-module $M$ is said to be \textbf{encoded} by $M':[m]\times [n]\rightarrow \vect$ if the following hold:
\begin{itemize}
    \item If $p\in \Zb^2$ is \emph{not} greater than equal to $(0,0)$, then $M(p)=0$.
    \item For $(0,0)\leq p$ in $\Zb^2$, we have that $M(p)=M'(q)$ where $q$ is the maximal element of $[m]\times [n]$ such that $q\leq p$ (we write $q=\lfloor p \rfloor_{m,n}$ in this case).
    \item 
For $(0,0)\leq p_1\leq p_2$ in $\Zb^2$, the  map $\varphi_M(p_1,p_2)$  is equal to $\varphi_{M'}(\lfloor p_1\rfloor_{m,n} \leq \lfloor p_2 \rfloor_{m,n})$.
\end{itemize}

\end{definition}
The $\Zb^2$-module $M$ described above is clearly finitely generated and its restriction $M\vert{}_{[m]\times [n]}$ coincides with $M'$.
The following proposition is the key to obtain Theorem \ref{thm:pictorial formula}.
Let $e_1:=(1,0)$ and $e_2:=(0,1)$ in $\Zb^2$. 

\begin{proposition}\label{prop:GPD determines Betti}
Assume that a $\Zb^2$-module $M$ is encoded by $M':[m]\times [n]\rightarrow \vect$.
Then, $\dgm(M')$ determines the bigraded Betti numbers of $M$ via the following formulas: For $p\notin [m+1]\times[n+1]$, we have $\beta_j(M)(p)=0$, $j=0,1,2$. For $p\in [m+1]\times [n+1]$, we have:

\begin{equation}\label{eq:Betti formulas}
\beta_{j}(M)(p)= \begin{cases} 
      \hspace{.8cm}\sum\limits_{\mathclap{\substack{J\ni p\\ J\niton p-e_1, p-e_2}
      }} \dgm(M')(J), & j=0  \\
    \hspace{1cm}\sum\limits_{\mathclap{\substack{J\ni p-e_1\\ J\niton p-e_1-e_2, p-e_2, p}}} \dgm(M')(J)  \hspace{.3cm}+  \hspace{.3cm} \sum\limits_{\mathclap{\substack{J \ni p-e_2\\ J\niton p-e_1-e_2, p-e_1, p}}} \dgm(M')(J) \hspace{.3cm}+  \hspace{.3cm}  \sum\limits_{\mathclap{\substack{J\ni p-e_1-e_2, p-e_1, p-e_2\\ J\niton p}}} \dgm(M')(J)  \\
      \hspace{0.8cm} + 2\hspace{.1cm} \sum\limits_{\mathclap{\substack{J\ni p-e_1, p-e_2\\ J\niton p-e_1-e_2, p}}} \dgm(M')(J)\hspace{.3cm}+\hspace{.3cm}\sum\limits_{\mathclap{\substack{J\ni p-e_1, p-e_2, p\\ J\niton p-e_1-e_2}}} \dgm(M')(J) \hspace{.3cm} -\hspace{.3cm}\sum\limits_{\mathclap{\substack{J\ni p-e_1-e_2, p-e_1, p\\ J\niton p-e_2}}} \dgm(M')(J) \\ \hspace{0.8cm} -\hspace{.3cm}\sum\limits_{\mathclap{\substack{J\ni p-e_1-e_2, p-e_2, p\\ J\niton p-e_1}}} \dgm(M')(J), &j=1 \\
      \hspace{.8cm}\sum\limits_{\mathclap{\substack{J\ni p-e_1-e_2\\ J\niton p-e_1, p-e_2}}} \dgm(M')(J),  & j=2,
   \end{cases}\end{equation} 
  where each sum is taken over $J\in \con([m]\times[n])$.\footnote{For example, when $j=0$, the sum is taken over every $J\in \con([m]\times[n])$ that contains $p$ and does not contain $p-e_1$ and $p-e_2$.}  Moreover, each $\dgm(M')$ above can be replaced by $\intdgm(M')$ where each sum is taken over $J\in\inter([m]\times [n])$.
\end{proposition}

One direct consequence of this proposition is that if $p\in[m+1]\times[n+1]$ is outside of $[m]\times [n]$, then $\beta_0(M)(p)=0$: no $J\in \con([m]\times [n])$ can include $p$ and thus the sum $\sum\limits_{\mathclap{\substack{J\ni p\\ J\niton p-e_1, p-e_2}
      }} \dgm(M')(J)$ is zero.
      We defer the proof of Proposition \ref{prop:GPD determines Betti} to the end of this section. 

\begin{remark}\label{rem: intervals}
In Proposition \ref{prop:GPD determines Betti}, the equation for $\beta_1(M)$ with respect to $\intdgm(M')$ can be  further simplified by removing the fourth, sixth, and seventh sums, i.e. 
\begin{multline*}
\beta_1(M)\\=\hspace{0.3cm}\sum\limits_{\mathclap{\substack{J\ni p-e_1\\ J\niton p-e_1-e_2, p-e_2, p}}} \intdgm(M')(J)  \hspace{.3cm}+  \hspace{.3cm} \sum\limits_{\mathclap{\substack{J \ni p-e_2\\ J\niton p-e_1-e_2, p-e_1, p}}} \intdgm(M')(J) \hspace{.3cm}+  \hspace{.3cm}  \sum\limits_{\mathclap{\substack{J\ni p-e_1-e_2, p-e_1, p-e_2\\ J\niton p}}} \intdgm(M')(J) 
+ \hspace{.3cm}  \sum\limits_{\mathclap{\substack{J\ni p-e_1, p-e_2, p\\ J\niton p-e_1-e_2}}} \intdgm(M')(J).
\end{multline*}
This is because the connected sets $J$ (Definition \ref{def:intervals} \ref{item:interval3}) 
over which the fourth, sixth, and seventh sums are taken cannot be intervals of $[m]\times[n]$ (those connected sets $J$
cannot satisfy Definition \ref{def:intervals} \ref{item:convexity}). Similarly, if $\dgm(M')(J)=0$ for all non-intervals $J\in \con([m]\times[n])$, then the fourth, sixth, and seventh sums can be eliminated in the equation for $\beta_1(M)$.\footnote{We remark that, in general, there can exist a non-interval $J\in \con([m]\times[n])$ where $\dgm(M')(J)\neq 0$; see Example \ref{ex:GPD_vs_IGPD} in the appendix.}
\end{remark}

By virtue of Remark \ref{rem: intervals}, Proposition \ref{prop:GPD determines Betti} admits a simple pictorial interpretation which generalizes Remark \ref{rem:bigraded bettis for an interval} \ref{item:bigraded bettis for an interval1} and Theorem \ref{thm:bigraded bettis for an interval2}. To state this interpretation, we introduce the following notation.

\begin{notation}\label{not:blow-up}
Given any $I\in \con(\Zb^2)$, let $I^+\subset \Rb^2$ be the corresponding region (cf. equation (\ref{eq:region})). Then $I^+$ admits the 3 types of corner points  depicted in Figure \ref{fig:corner points in general}. For $j=0,1,2$, we define functions $\tau_j(I^+):\Zb^2\rightarrow \{0,1,2\}$ as follows: for $j=0,2$, let $\tau_j(I^+)(p):=1$ if $p$ is a $j^\mathrm{th}$ type corner point of $I^+$, and $0$ otherwise. For $j=1$, let
\[\tau_1(I^+)(p):=\begin{cases}2, &\mbox{$p$ is a $1^\mathrm{st}$-type corner point of $I^+$ with multiplicity 2}\\1, &\mbox{$p$ is a $1^\mathrm{st}$-type corner point of $I^+$ with multiplicity 1}\\0,&\mbox{otherwise.}
\end{cases}\]
\end{notation}

\begin{figure}
    \centering
    \includegraphics[width=0.55\textwidth]{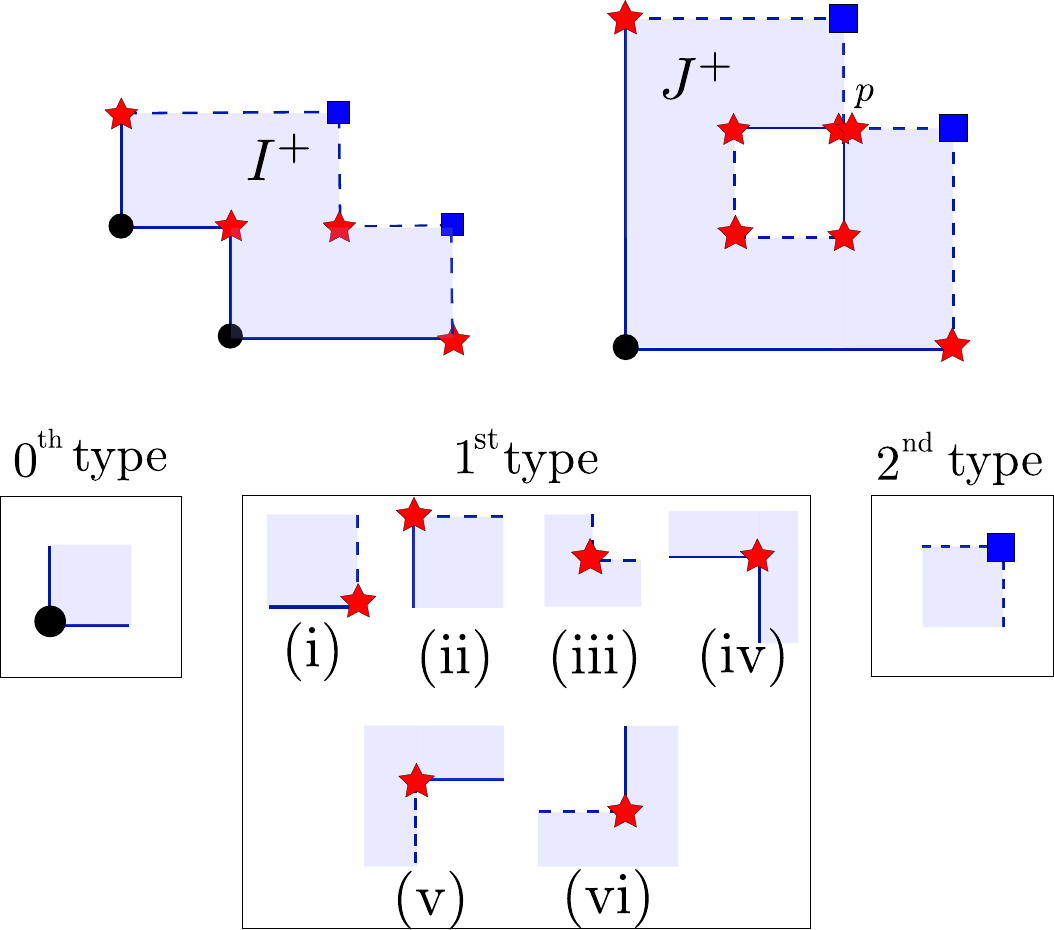}
    \caption{
    The three different types of corner points in $I^+\subset \Rb^2$ and $J^+\subset \Rb^2$. 
    Note that two different $1^\mathrm{st}$ type corner points of $J$ are located at $p$. See Definition \ref{def:types of corners} for a rigorous description of each of the three types of corner points.}
    \label{fig:corner points in general}
\end{figure}

Our main theorem below says that the bigraded Betti numbers of a given $\Zb^2$-module $M$ encoded by an $([m]\times [n])$-module $M'$ can be read off from the corner points of the elements in either of \[\{ I^+\subset\Rb^2: \dgm(M')(I)\neq 0\} \ \mbox{and}\ \{ I^+\subset \Rb^2: \intdgm(M')(I)\neq 0\}.\]

\begin{theorem}\label{thm:pictorial formula}
Assume that a $\Zb^2$-module $M$ is encoded by $M':[m]\times [n]\rightarrow \vect$. Then, for every $j=0,1,2$ and for every $p\in\Zb^2$, we have  
\begin{equation}\label{eq:pictorial}
    \beta_j(M)(p)=\sum\limits_{I\in\con([m]\times[n])}\dgm(M')(I)\times\tau_j(I^+)(p) .
\end{equation}
Also we have:
\begin{equation}\label{eq:pictorial2}
    \beta_j(M)(p)=\sum\limits_{I\in\inter([m]\times[n])}\intdgm(M')(I)\times\tau_j(I^+)(p) .
\end{equation}
\end{theorem}
Notice that, by Theorem \ref{thm:GPD is barcode proxy}, the theorem above is a generalization of Theorem \ref{thm:bigraded bettis for an interval2}. We prove Theorem \ref{thm:pictorial formula} at the end of this section.

\begin{example}\label{ex:pictorial interpretation}
Recall that $[3]\times[2]=\{0,1,2,3\}\times\{0,1,2\}\subset \Zb^2$ and assume that a $\Zb^2$-module $N$ is encoded by the module $N':[3]\times[2]\rightarrow \vect$ depicted in Fig.~\ref{fig:introduction} (A'). Then, Fig.~\ref{fig:introduction} (B') and (C') are explained as follows:  
\begin{itemize}
    \item[(B')] For $I_1,I_2,I_3,I_4\in \inter([3]\times[2])$ in Fig.~\ref{fig:intervals}, we have that $\intdgm(N')(I_i)=1$ for $i=1,2,3$ and $\intdgm(N')(I_4)=-1$ and $\intdgm(N')(J)=0$ for the other $J\in\inter([3]\times[2])$ (more details are provided after this example). 
    \item[(C')]  
    For $i=1,2,3,4$, expand $I_i$ to its corresponding region $I_i^+$ in $\Rb^2$ (cf. Definition \ref{def:corresponding region}).    The corner points of each $I_i^+$ are marked according to their types as described in Fig. \ref{fig:corner points in general}. By Theorem \ref{thm:pictorial formula}, for each $p\in \Zb^2$ and $j=0,1,2$, $\beta_j(N)(p)$ is equal to the number of black dots, red stars, and blue squares at $p$ respectively, where the corner points of the red interval $I_4^+$ negatively contribute to the counting. The net sum  is illustrated in Fig.~\ref{fig:Bettis_onemore}.
\end{itemize}
\end{example}

\begin{figure}
    \centering
    \includegraphics[width=0.8\textwidth]{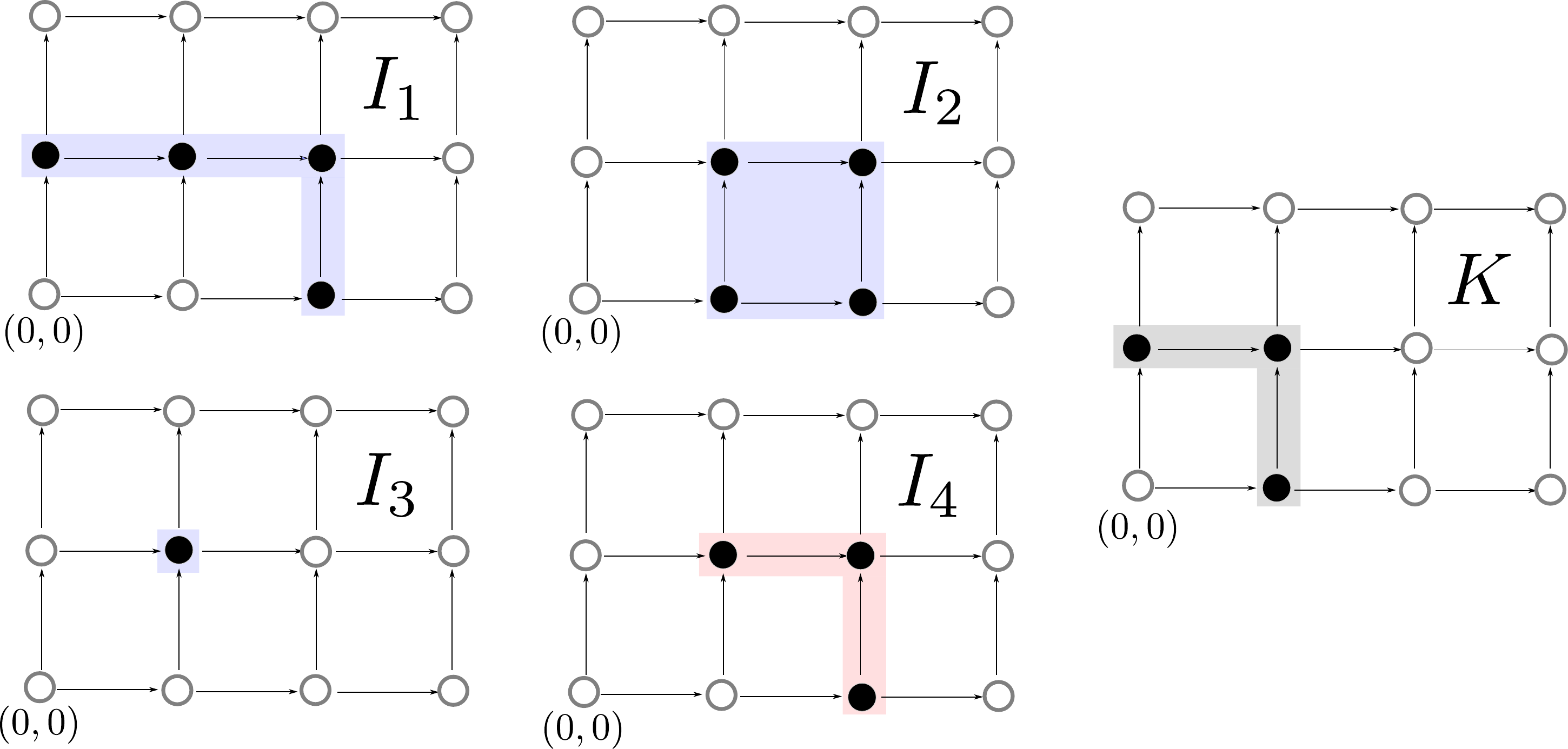}
    \caption{$I_1$, $I_2$, $I_3$, and $I_4$ are the intervals corresponding to Fig.~\ref{fig:introduction} (B').}
    \label{fig:intervals}
\end{figure}

\begin{figure}
    \centering
    \includegraphics[width=0.3\textwidth]{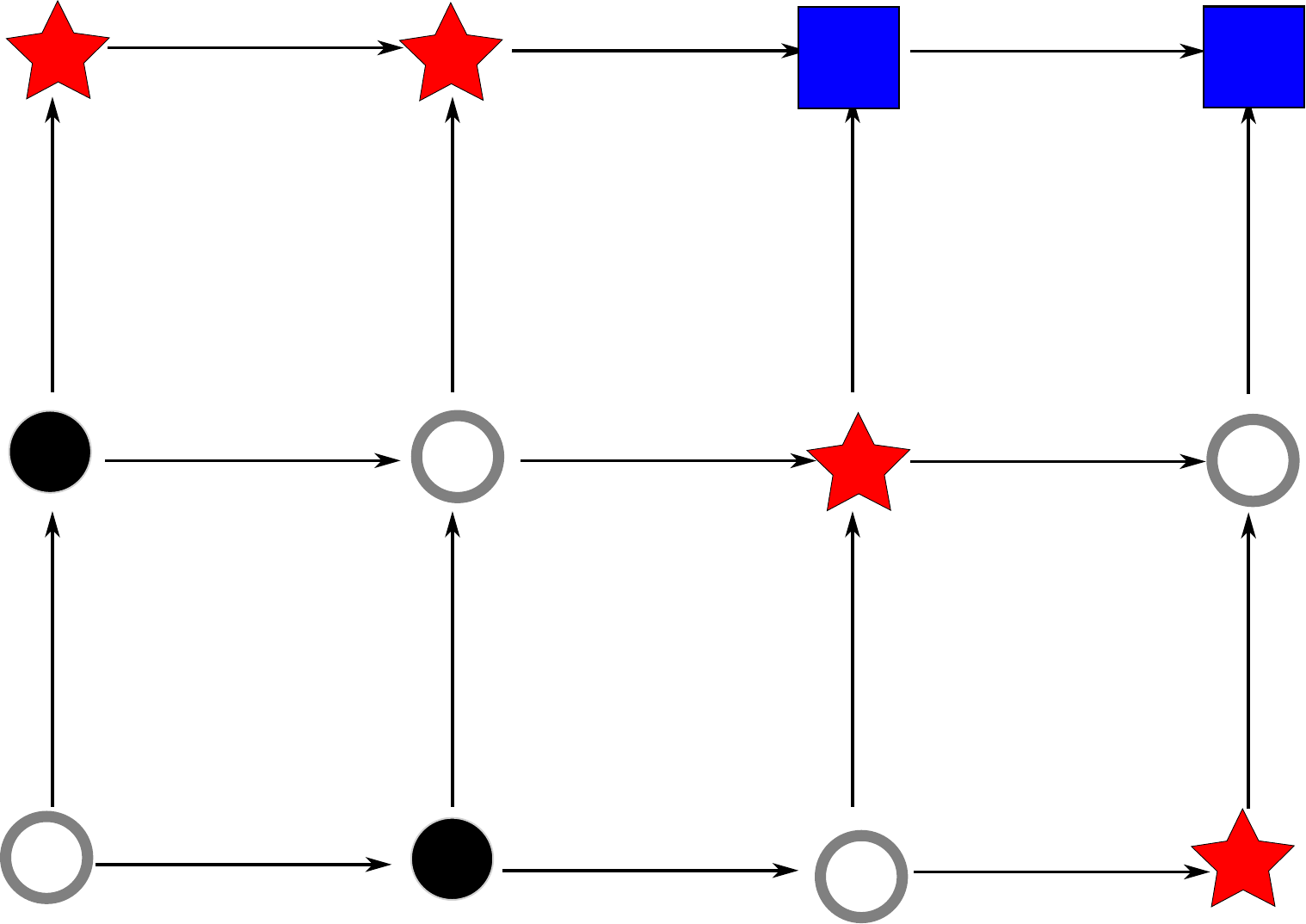}
    \caption{For $N:\Zb^2\rightarrow \vect$ in Example \ref{ex:pictorial interpretation}, a black dot at $p$ indicates $\beta_0(N)(p)=1$, a red star at $p$ indicates $\beta_1(N)(p)=1$, and a blue square at $p$ indicates $\beta_2(N)(p)=1$. For all other $j$ and $p$, $\beta_j(N)(p)=0$.}
    \label{fig:Bettis_onemore}
\end{figure}

\begin{proof}[Details about Example \ref{ex:pictorial interpretation}]
For $I_1,I_2,I_3,I_4\in \inter([3]\times[2])$ in Fig.~\ref{fig:intervals}, we show that $\intdgm(N')(I_i)=1$ for $i=1,2,3$ and $\intdgm(N')(I_4)=-1$ and $\intdgm(N')(J)=0$ for the other $J\in\inter([3]\times[2])$. 

\begin{enumerate}[label=(\roman*)] 
    \item \label{item:computation1}If $J\in \inter([2]\times [3])$ contains any point $p\in [3]\times[2]$ such that $N'(p)=0$, then $\intdgm(N')(J)=0$ by Remarks \ref{rem:rank invariant properties} \ref{item:rank invariant properties2} and \ref{rem:rank invariant properties2}.
    \item \label{item:computation2}Consider $K:=\{(0,1),(1,1),(1,0)\}\in  \inter([3]\times[2])$ that is depicted in Fig.~\ref{fig:intervals}.  We claim that for all $J\supseteq K$, $\intdgm(N')(J)=0$: By Remarks \ref{rem:rank invariant properties} \ref{item:rank invariant properties1} and \ref{rem:rank invariant properties2}, it suffices to show that $\intrk(N')(K)=0$. This follows from Theorem \ref{thm:rkequalsintervals} and the fact that the zigzag module $N'\vert{}_K$ does not admit a summand that is isomorphic to the interval module $V_{K}:K\rightarrow \vect$. An alternative way to prove $\intrk(N')(K)=0$ is to show that $\varprojlim N'\vert{}_K$ is trivial: Note that 
        $\varprojlim N'\vert{}_K\cong(L,\{\pi_p\}_{p\in K})$, where 
    \begin{multline*}
        L= \{(v_1,v_2,v_3)\in N'{(0,1)}\oplus N'{(1,1)}\oplus N'{(1,0)}:\\ \varphi_{N'}((0,1),(1,1))(v_1)=v_2=\varphi_{N'}((1,0),(1,1))(v_3)\}
    \end{multline*}
and $\pi_p:L\rightarrow N'(p)$ are the canonical projections for $p\in K$. Then, we have:
\begin{align*}
    L&=\{(x_1,(x_2,x_3),x_4)\in \F\oplus(\F^2)\oplus\F:x_1=x_2,\ x_3=0,\ x_2=x_3=x_4\}\\&=\{(0,(0,0),0)\}.
\end{align*}

\item We claim that $\intdgm(N')(I_1)=\intdgm(N')(I_2)=1$. Fix any $i\in\{1,2\}$. By invoking Theorem \ref{thm:rkequalsintervals}, one can check that $\intrk(N')(I_i)=\rk(N'\vert{}_{I_i})=1$. 
Let us observe that any interval $J\supsetneq I_i$ must contain either $K$ or a point $p\in[3]\times[2]$ such that $N'(p)=0$. Hence, by Remarks \ref{rem:rank invariant properties} \ref{item:rank invariant properties2} and \ref{item:rank invariant properties1}, we have that $\intrk(N')(J)=0$. Therefore, by Theorem \ref{thm:mobius inversion by asashiba}, we have: \begin{multline*}\intdgm(N')(I_i)=\intrk(N')(I_i)+\sum_{\substack{S\subseteq \cov(I_i)\\ S\neq \emptyset}}(-1)^{\abs{S}}\intrk(N')\left(\bigvee S\right)\\=1+\sum_{\substack{S\subseteq \cov(I_i)\\ S\neq \emptyset}}(-1)^{\abs{S}}\cdot 0=1.
\end{multline*}
\end{enumerate}
Similarly, one can compute $\intdgm(N')(I_3)=1$, $\intdgm(N')(I_4)=-1$, and $\intdgm(N')(L)=0$ for any $L\in \inter([3]\times[2])$ that has not been considered so far.
\end{proof}

\paragraph{Proofs of Proposition \ref{prop:GPD determines Betti} and Theorem \ref{thm:pictorial formula}.}
Lemma \ref{lem:Combinatorial Betti Formula} below will be used in the proof of Proposition \ref{prop:GPD determines Betti}. Let $M$ be any finitely generated $\Zb^2$-module.
For any $p\in\Zb^2$, consider the subposet $\{p-e_1\leq p \geq  p-e_2\}$
where $e_1=(1,0)$ and $e_2=(0,1)$. The restriction of $M$ to $\{p-e_1\leq p \geq p-e_2\}$ is a zigzag module and thus admits a barcode (Theorem \ref{thm:interval decomposable}). Let $n_p$ be the 
multiplicity of $\{p\}$
in the barcode of $M\vert{}_{\{p-e_1\leq p \geq p-e_2\}}$. Similarly,  we also consider the subposet $\{p+e_1\geq p \leq  p+e_2\}$ 
and define $m_p$ to be the multiplicity of $\{p\}$ in the barcode of $M\vert{}_{\{p+e_1 \geq p \leq p+e_2\}}$.

\begin{lemma}[\cite{moore2020subquivers,moore2022structure}]\label{lem:Combinatorial Betti Formula} 
Given any finitely generated $\Zb^2$-module $M$, for every $p\in\Zb^2$, we have:
\begin{equation}\beta_{j}(M)(p)= \begin{cases} 
      n_p & j=0 \\
      n_p-\dim(M(p))+\dim(M(p-e_1))+\dim(M(p-e_2))\\
      \hspace{2cm} -\dim(M(p-e_1-e_2))+m_{p-e_1-e_2} & j=1 \\
      m_{p-e_1-e_2} & j=2.
  \end{cases}\end{equation}
\end{lemma}

A combinatorial proof of this lemma 
can be found in \cite[Corollary 2.3]{moore2020subquivers}. This lemma can be also proved by utilizing machinery from commutative algebra as follows (see {\cite[Section 2A.3]{eisenbudbetti}} for details): A finitely generated 2-parameter persistence module $M$ can equivalently be considered as an $\N^2$ graded module over $k[x_1,x_2]$.  The bigraded Betti numbers of $M$ can be defined using tensor products, after which Lemma \ref{lem:Combinatorial Betti Formula} follows by tensoring $M$ with the Koszul complex on $x_1$ and $x_2$. 

\new{For $p\in \Zb^2$, we consider the subposets $\llcorner p:=\{p+e_1> p < p+e_2\}$ and $p\urcorner:=\{p-e_1< p > p-e_2\}$ of $\Zb^2$. Let $\mathbf{M}(\Zb^2)$ be the collection of all multisets of subsets of $\Zb^2$.  Lemma \ref{lem:Combinatorial Betti Formula} implies:}
\begin{corollary}\label{cor:zz barcode determine bettis} \new{Let $M$ be any finitely generated $\Zb^2$-module. Then, the two maps $\Zb^2\rightarrow \mathbf{M}(\Zb^2)$ given by
\[p\mapsto \barc(M\vert_{\llcorner p})\ \ \mbox{and}\ \ p\mapsto \barc(M\vert_{p\urcorner})\]
determine the bigraded Betti numbers of $M$.} 
\end{corollary}

\begin{proof}[Proof of Proposition \ref{prop:GPD determines Betti}]
We consider the case $j=1,$ as the other cases are similar.
By Lemma \ref{lem:Combinatorial Betti Formula},
\begin{multline}\label{eq:6}\beta_1(M)(p)=\\ n_p-\dim(M(p))+\dim(M(p-e_1))+\dim(M(p-e_2))-\dim(M(p-e_1-e_2))+m_{p-e_1-e_2}.
\end{multline}
Let $p\in \Zb^2$ where $p\notin[m]\times [n]$. Then, we claim that $\beta_1(M)(p)=0$. This fact can be shown by checking that $0=n_p=m_{p-e_1-e_2}$ and \[0=-\dim(M(p))+\dim(M(p-e_1))+\dim(M(p-e_2))-\dim(M(p-e_1-e_2)).\]
Next, let $p\in [m]\times [n]$.
We will now find a formula for each term in the right-hand side (RHS) of Equation (\ref{eq:6}) in terms of the generalized rank invariant of $M\vert{}_{[m]\times [n]}=M'$. Notice that for every $q\in[m]\times [n]$,
\begin{equation}\label{eq:7}
    \dim(M(q))=\rk(M')(\{q\}).
\end{equation}
Next, consider $M\vert{}_{\{p-e_1\leq p \geq p-e_2\}}$, which is a zigzag module and thus it is interval decomposable (Theorem \ref{thm:interval decomposable}). Recall that $n_p$ is the multiplicity of $\{p\}$ in the barcode of $M\vert{}_{\{p-e_1\leq p \geq p-e_2\}}$. From Example \ref{ex:zigzag mobius inversion}, we know that:
\begin{equation}\label{eq:8}
n_p=\rk(M')(\{p\})-\rk(M')(\{p-e_1\leq p\})-\rk(M')(\{p \geq p-e_2\})+\rk(M')(\{p-e_1\leq p \geq p-e_2\}).\end{equation}
Similarly, we have:
\begin{equation}\label{eq:9}
  \begin{aligned}
   m_{p-e_1-e_2}=&\rk(M')(\{p-e_1-e_2\})-\rk(M')(\{p-e_1-e_2\leq p-e_1\}) \\
    &- \rk(M')(\{p-e_1-e_2 \leq p-e_2\}) +\rk(M')(\{p-e_1\geq p-e_1-e_2 \leq p-e_2\}). \\
  \end{aligned}
\end{equation}
Combining equations (\ref{eq:6}), (\ref{eq:7}), (\ref{eq:8}), (\ref{eq:9}) yields:
\begin{equation*}
  \begin{aligned}
  \beta_1(M)(p)= &-\rk(M')(\{p-e_1\leq p\})-\rk(M')(\{p \geq p-e_2\})+\rk(M')(\{p-e_1\leq p \geq p-e_2\})\\
  &+\rk(M')(\{p-e_1\})+\rk(M')(\{p-e_2\})-\rk(M')(\{p-e_1-e_2\leq p-e_1\})\\
  &- \rk(M')(\{p-e_1-e_2 \leq p-e_2\})+\rk(M')(\{p-e_1\geq p-e_1-e_2 \leq p-e_2\}). \\
  \end{aligned}
\end{equation*}
Since $p\in [m]\times [n]$, by invoking Equation (\ref{eq:GPD0}), we obtain:
\begin{equation}\label{eq:10}
  \begin{aligned}
  \beta_1(M)(p)=&-\hspace{.3cm}\sum\limits_{\mathclap{J\ni p-e_1,p}} \dgm(M')(J) \hspace{.3cm}-\hspace{.3cm}\sum\limits_{\mathclap{J\ni p-e_2, p}} \dgm(M')(J) \hspace{.3cm}+ \hspace{.3cm} \sum\limits_{\mathclap{J\ni p-e_1, p-e_2, p}} \dgm(M')(J)\\ 
  &+\hspace{.3cm} \sum\limits_{\mathclap{J\ni p-e_1}} \dgm(M')(J) \hspace{.3cm} +\hspace{.3cm} \sum\limits_{\mathclap{J\ni p-e_2}} \dgm(M')(J)\hspace{.3cm}-\hspace{.3cm} \sum\limits_{\mathclap{J\ni p-e_1-e_2, p-e_1}} \dgm(M')(J)\\ &-\hspace{.3cm}\sum\limits_{\mathclap{J\ni p-e_1-e_2, p-e_2}} \dgm(M')(J) \hspace{.3cm}+ \hspace{.3cm} \sum\limits_{\mathclap{J\ni p-e_1-e_2, p-e_1, p-e_2}} \dgm(M')(J)\\
  \end{aligned}
\end{equation}
Let $J\in\con([m]\times [n])$. The multiplicity of $\dgm(M')(J)$ in the RHS of Equation (\ref{eq:10}) is fully determined by the intersection of $J$ and the four-point set $\{p-e_1-e_2,\ p-e_1,\ p-e_2,\ p\}$. For example, if $p-e_1, p-e_1-e_2\in J$ and $p-e_2, p\notin J,$ then $\dgm(M')(J)$ occurs only in the fourth and sixth sums, and has an overall multiplicity of zero in the RHS of Equation (\ref{eq:10}). For another example, if $p-e_1\in J$ 
and $p-e_1-e_2, p-e_2, p\notin J$, then $\dgm(M')(J)$ occurs only in the fourth summand, which yields the first sum of the RHS in Equation (\ref{eq:11}) below. Considering all possible $2^4$ combinations of the intersection of $J$ and the four-point set $\{p-e_1-e_2,\ p-e_1,\ p-e_2,\ p\}$ yields

\begin{equation}\label{eq:11}
  \begin{aligned}
  \beta_1(M)(p)&=\hspace{.3cm}\sum\limits_{\mathclap{\substack{J\ni p-e_1\\ J\niton p-e_1-e_2, p-e_2, p}}} \dgm(M')(J)  \hspace{.3cm}+  \hspace{.3cm} \sum\limits_{\mathclap{\substack{J \ni p-e_2\\ J\niton p-e_1-e_2, p-e_1, p}}} \dgm(M')(J) \hspace{.3cm}+  \hspace{.3cm}  \sum\limits_{\mathclap{\substack{J\ni p-e_1-e_2, p-e_1, p-e_2\\ J\niton p}}} \dgm(M')(J)  \\
      & +\hspace{.1cm}2\hspace{.1cm} \sum\limits_{\mathclap{\substack{J\ni p-e_1, p-e_2\\ J\niton p-e_1-e_2, p}}} \dgm(M')(J)\hspace{.3cm}+\hspace{.3cm}\sum\limits_{\mathclap{\substack{J\ni p-e_1, p-e_2, p\\ J\niton p-e_1-e_2}}} \dgm(M')(J) \hspace{.3cm} -\hspace{.3cm}\sum\limits_{\mathclap{\substack{J\ni p-e_1-e_2, p-e_1, p\\ J\niton p-e_2}}} \dgm(M')(J) \\ &-\hspace{.3cm}\sum\limits_{\mathclap{\substack{J\ni p-e_1-e_2, p-e_2, p\\ J\niton p-e_1}}} \dgm(M')(J), 
  \end{aligned}
\end{equation}
as claimed.
\end{proof}

\begin{proof}[Proof of Theorem \ref{thm:pictorial formula}]
We only prove equation (\ref{eq:pictorial2}) with $j=1$, as the other cases are similar. By Remark \ref{rem: intervals},
\begin{multline}\label{eq:12}
  \beta_1(M)(p)=\\ \hspace{.3cm}\sum\limits_{\mathclap{\substack{J\ni p-e_1\\ J\niton p-e_1-e_2, p-e_2, p}}} \intdgm(M')(J)  \hspace{.3cm}+  \hspace{.3cm} \sum\limits_{\mathclap{\substack{J \ni p-e_2\\ J\niton p-e_1-e_2, p-e_1, p}}} \intdgm(M')(J) \hspace{.3cm}+  \hspace{.3cm}  \sum\limits_{\mathclap{\substack{J\ni p-e_1-e_2, p-e_1, p-e_2\\ J\niton p}}} \intdgm(M')(J)
  + \hspace{.3cm} \sum\limits_{\mathclap{\substack{J\ni p-e_1, p-e_2, p\\ J\niton p-e_1-e_2}}} \intdgm(M')(J),
\end{multline}
where each sum is taken over $J\in\inter(\Zb^2)$.

Furthermore, for $I\in\inter(\Zb^2)$, observe that $\tau_1(I^+)(p)=1$ if and only if one of the following is true about $I$: (i) $p-e_1\in I$ and $p-e_1-e_2, p-e_2, p\notin I$, (ii) $p-e_2\in I$ and $p-e_1-e_2, p-e_1, p\notin I$, (iii) $p-e_1-e_2, p-e_1, p-e_2\in I$ and $p\notin I$, or (iv) $p-e_1, p-e_2, p\in I$ and $p-e_1-e_2\notin I$. Otherwise, $\tau_1(I^+)(p)=0$. These four cases (i),(ii),(iii), and (iv) correspond to the four sums on the RHS of Equation (\ref{eq:12})  in order, and also correspond to the $1^{\mathrm{st}}$ corner types (i),(ii),(iii), (iv) given in Figure \ref{fig:corner points in general}. Therefore:
\begin{equation*}\label{eq:13}
  \begin{aligned}
  \beta_1(M)(p)=
  \sum\limits_{I\in\inter([m]\times[n])}\intdgm(M')(I)\times\tau_1(I^+)(p).\\
  \end{aligned}
\end{equation*}
\end{proof}

\section{Impossibility of extending Theorem \ref{thm:pictorial formula} to $d$-parameter persistence modules for $d\geq 3$}

In this section, we show that Theorem \ref{thm:pictorial formula} cannot be extended to $\Zb^d$-modules for $d\geq 3$. This impossibility claim directly follows from:

\begin{theorem}\label{thm:no pictorial formula}
For $d\geq 3$, the generalized persistence diagram does not determine the multigraded Betti numbers of $\Zb^d$-modules. In particular, there exists a pair of $\Zb^d$-modules that have the same generalized persistence diagram, but different multigraded Betti numbers.
\end{theorem}

To prove this theorem, we find a pair of $\Zb^3$-modules that have the same generalized persistence diagram, but different multigraded Betti numbers. Since any $\Zb^d$-module $M$ can be trivially extended to the $\Zb^{d+1}$-module $M\times0$, the existence of such a pair proves the claim for arbitrary $d \geq 3$.

\begin{proof}[Proof of Theorem \ref{thm:no pictorial formula}.]
Let $\pi_1,\pi_2:k^2\rightarrow k$ be the canonical projections onto the first and the second coordinates of $k^2$ respectively. Let $M$ be the $\Zb^3$-module described as follows: Over the subset $\{(x,y,z)\in\Zb^3: 0\leq x,y,z\leq 1\}\subset \Zb^3$, $M$ is given by 
\begin{center}
\tdplotsetmaincoords{60}{125}
\tdplotsetrotatedcoords{0}{30}{120} 
\begin{tikzpicture}[scale=2,tdplot_rotated_coords,
                    cube/.style={very thick,black},
                    grid/.style={very thin,gray},
                    axis/.style={->,blue,ultra thick},
                    rotated axis/.style={->,purple,ultra thick}]
\foreach \x in {0,1}
   \foreach \y in {0,1}
      \foreach \z in {0,1}{
           \ifthenelse{\x<1 }
           {
                \draw [black]   (\x+0.2,\y,\z) -- (\x+0.8,\y,\z);
           }
           {
           }
           \ifthenelse{ \y<1  }
           {
                \draw [black]   (\x,\y+0.2,\z)-- (\x,\y+0.8,\z);
           }
           {
           }
           \ifthenelse{ \z<1  }
           {
                \draw [black]   (\x,\y,\z+0.2) -- (\x,\y,\z+0.8);
           }
           {
           }
          \node[scale=0.7] at (0,0,0) {$(0,0,0)$};
           \node[scale=0.7] at (0,0,1) {$(0,1,0)$};       
            \node[scale=0.7] at (1,0,0) {$(1,0,0)$};
            \node[scale=0.7] at (0,1,0) {$(0,0,1)$}; 
            \node[scale=0.7] at (1,1,1) {$(1,1,1)$};
              \node[scale=0.7] at (1,1,0) {$(1,0,1)$};
                \node[scale=0.7] at (1,0,1) {$(1,1,0)$};  
               \node[scale=0.7] at (0,1,1) {$(0,1,1)$}; 
}

\end{tikzpicture}
\hspace{10mm}
\begin{tikzpicture}[scale=2,tdplot_rotated_coords,
                    cube/.style={very thick,black},
                    grid/.style={very thin,gray},
                    axis/.style={->,blue,ultra thick},
                    rotated axis/.style={->,purple,ultra thick}]
\foreach \x in {0,1}
   \foreach \y in {0,1}
      \foreach \z in {0,1}{
           \ifthenelse{\x<1 }
           {
                \draw [->,black]   (\x+0.15,\y,\z) -- (\x+0.85,\y,\z);
           }
           {
           }
           \ifthenelse{ \y<1  }
           {
                \draw [->,black]   (\x,\y+0.15,\z)-- (\x,\y+0.85,\z);
           }
           {
           }
           \ifthenelse{ \z<1  }
           {
                \draw [->,black]   (\x,\y,\z+0.15) -- (\x,\y,\z+0.85);
           }
           {
           }}
         \node[scale=1] at (0,0,0) {$k^2$};
            \node at (0,0,1) {$k$};
        \node at (1,0,0) {$k$};
                   \node at (0,1,0) {$k$};
            \node[scale=0.9] at (1,1,1) {$0$};
              \node[scale=0.9] at (1,1,0) {$0$};
                \node[scale=0.9] at (1,0,1) {$0$};
                  \node[scale=0.9] at (1,1,0) {$0$};                  
                  \node[scale=0.9] at (0,1,1) {$0$};
                \draw[->,opacity=0] (0,0,0) -- (0,1,0) node[pos=0.5,right,opacity=1] {$\pi_1+\pi_2$};       
 \draw[->,opacity=0] (0,0,0) -- (0,0,1) node[midway,left,opacity=1] {$\pi_2$};      
  \draw[->,opacity=0] (0,0,0) -- (1,0,0) node[pos=0.6,below,opacity=1] {$\pi_1$};      

\end{tikzpicture}
\end{center}
and $M(p)=0$ for $p\not\in \{(x,y,z)\in\Zb^3: 0\leq x,y,z\leq 1\}$. In the poset $\Zb^3$, consider the intervals \[I_0:=\{(0,0,0)\},\ \ I_1:=I_0\cup\{(1,0,0)\},\ \ I_2:=I_0\cup\{(0,1,0)\},\ \ I_3:=I_0\cup\{(0,0,1)\}. \] 
Combine the interval modules $V_{I_i}$ for $i=0,1,2,3$ and $M$ into the modules \[N_1:=M\oplus V_{I_0} \ \ \mbox{and}\ \ \ N_2:=V_{I_1}\oplus V_{I_2} \oplus V_{I_3}.\] Since $N_2$ is interval decomposable, by Theorem \ref{thm:GPD is barcode proxy}, both $\dgm(N_2)$ and $\intdgm(N_2)$ amount to the barcode of $N_2$. We observe that $\rk(N_1)=\rk(N_2)$, which implies not only $\intrk(N_1)=\intrk(N_2)$, but also \[\dgm(N_1)=\dgm(N_2) \ \ \mbox{and}\ \ \intdgm(N_1)=\intdgm(N_2).\]

Now we prove that the multigraded Betti numbers of $N_1$ and $N_2$ do not coincide. Since multigraded Betti numbers are additive (Remark \ref{rem:bigraded bettis for an interval} \ref{item:bigraded betti is additive}), we have that 
\begin{equation*}
\beta_3(N_1)(1,1,1)= \beta_3(M)(1,1,1)+\beta_3(V_{I_0})(1,1,1)\geq \beta_3(V_{I_0})(1,1,1)=1,
\end{equation*}
as seen in Figure \ref{fig:3D}. However, since $\beta_3(V_{I_1})(1,1,1)=\beta_3(V_{I_2})(1,1,1)=\beta_3(V_{I_3})(1,1,1)=0$ (also shown in Figure \ref{fig:3D}), we have that $\beta_3(N_2)(1,1,1)=0$, completing the proof. 
 \end{proof}

\tdplotsetmaincoords{60}{125}
\tdplotsetrotatedcoords{0}{30}{120} 
\begin{figure}
\begin{center}
\begin{tikzpicture}[scale=1,tdplot_rotated_coords,
                    cube/.style={very thick,black},
                    grid/.style={very thin,gray},
                    axis/.style={->,blue,ultra thick},
                    rotated axis/.style={->,purple,ultra thick}]
\filldraw[fill=blue,opacity=0.3] (0,0,0) -- (1,0,0) -- (1,1,0) -- (1,1,1) -- (0,1,1) -- (0,0,1) -- (0,0,0) -- cycle;  
\foreach \x in {0,1,2}
   \foreach \y in {0,1,2}
      \foreach \z in {0,1,2}{
           \ifthenelse{  \lengthtest{\x pt < 2pt}  }
           {
                \draw [black]   (\x,\y,\z) -- (\x+1,\y,\z);
           }
           {
           }
           \ifthenelse{  \lengthtest{\y pt < 2pt}  }
           {
                \draw [black]   (\x,\y,\z)-- (\x,\y+1,\z);
           }
           {
           }
           \ifthenelse{  \lengthtest{\z pt < 2pt}  }
           {
                \draw [black]   (\x,\y,\z) -- (\x,\y,\z+1);
           }
           {
           }
         \node[scale=1] at (0,0,0) {$\bullet$};
          \node[scale=0.7,below] at (0,0,0) {$(0,0,0)$};
           \node[scale=0.7,left] at (0,0,1) {$(0,1,0)$};
            \node[scale=0.7,below] at (1,0,0) {$(1,0,0)$};
            \node[scale=0.9] at (1,1,1) {$\blacktriangle$};
            \node[scale=0.9,red] at (1,0,0) {$\bigstar$};
             \node[scale=0.9,red] at (0,1,0) {$\bigstar$};
              \node[scale=0.9,red] at (0,0,1) {$\bigstar$};
               \node[scale=0.9,blue] at (1,1,0) {$\blacksquare$};
                \node[scale=0.9,blue] at (1,0,1) {$\blacksquare$};
                  \node[scale=0.9,blue] at (0,1,1) {$\blacksquare$};
}

\end{tikzpicture}
\begin{tikzpicture}[scale=1,tdplot_rotated_coords,
                    cube/.style={very thick,black},
                    grid/.style={very thin,gray},
                    axis/.style={->,blue,ultra thick},
                    rotated axis/.style={->,purple,ultra thick}]
\filldraw[fill=blue,opacity=0.3] (0,0,0) -- (2,0,0) -- (2,1,0) -- (2,1,1) -- (0,1,1) -- (0,0,1) -- (0,0,0) -- cycle;                    
\foreach \x in {0,1,2}
   \foreach \y in {0,1,2}
      \foreach \z in {0,1,2}{
           \ifthenelse{  \lengthtest{\x pt < 2pt}  }
           {
                \draw [black]   (\x,\y,\z) -- (\x+1,\y,\z);
           }
           {
           }
           \ifthenelse{  \lengthtest{\y pt < 2pt}  }
           {
                \draw [black]   (\x,\y,\z)-- (\x,\y+1,\z);
           }
           {
           }
           \ifthenelse{  \lengthtest{\z pt < 2pt}  }
           {
                \draw [black]   (\x,\y,\z) -- (\x,\y,\z+1);
           }
           {
           }
         \node[scale=1] at (0,0,0) {$\bullet$};
            \node[scale=0.9] at (2,1,1) {$\blacktriangle$};
             \node[scale=0.9,red] at (2,0,0) {$\bigstar$};
             \node[scale=0.9,red] at (0,1,0) {$\bigstar$};
              \node[scale=0.9,red] at (0,0,1) {$\bigstar$};
               \node[scale=0.9,blue] at (2,1,0) {$\blacksquare$};
                \node[scale=0.9,blue] at (2,0,1) {$\blacksquare$};
                  \node[scale=0.9,blue] at (0,1,1) {$\blacksquare$};
}

\end{tikzpicture}
\begin{tikzpicture}[scale=1,tdplot_rotated_coords,
                    cube/.style={very thick,black},
                    grid/.style={very thin,gray},
                    axis/.style={->,blue,ultra thick},
                    rotated axis/.style={->,purple,ultra thick}]
\filldraw[fill=blue,opacity=0.3] (0,0,0) -- (1,0,0) -- (1,1,0) -- (1,1,2) -- (0,1,2) -- (0,0,2) -- (0,0,0) -- cycle;
\foreach \x in {0,1,2}
   \foreach \y in {0,1,2}
      \foreach \z in {0,1,2}{
           \ifthenelse{  \lengthtest{\x pt < 2pt}  }
           {
                \draw[black]   (\x,\y,\z) -- (\x+1,\y,\z);
           }
           {
           }
           \ifthenelse{  \lengthtest{\y pt < 2pt}  }
           {
                \draw[black]   (\x,\y,\z)-- (\x,\y+1,\z);
           }
           {
           }
           \ifthenelse{  \lengthtest{\z pt < 2pt}  }
           {
                \draw[black]   (\x,\y,\z) -- (\x,\y,\z+1);
           }
           {
           }
         \node[scale=1] at (0,0,0) {$\bullet$};
            \node[scale=0.9] at (1,1,2) {$\blacktriangle$};
          \node[scale=0.9,red] at (1,0,0) {$\bigstar$};
             \node[scale=0.9,red] at (0,1,0) {$\bigstar$};
              \node[scale=0.9,red] at (0,0,2) {$\bigstar$};
               \node[scale=0.9,blue] at (1,1,0) {$\blacksquare$};
                \node[scale=0.9,blue] at (1,0,2) {$\blacksquare$};
                  \node[scale=0.9,blue] at (0,1,2) {$\blacksquare$};   
}
\end{tikzpicture}
\begin{tikzpicture}[scale=1,tdplot_rotated_coords,
                    cube/.style={very thick,black},
                    grid/.style={very thin,gray},
                    axis/.style={->,blue,ultra thick},
                    rotated axis/.style={->,purple,ultra thick}]
\filldraw[fill=blue,opacity=0.3] (0,0,0) -- (1,0,0) -- (1,2,0) -- (1,2,1) -- (0,2,1) -- (0,0,1) -- (0,0,0) -- cycle;                    
\foreach \x in {0,1,2}
   \foreach \y in {0,1,2}
      \foreach \z in {0,1,2}{
           \ifthenelse{  \lengthtest{\x pt < 2pt}  }
           {
                \draw [black]   (\x,\y,\z) -- (\x+1,\y,\z);
           }
           {
           }
           \ifthenelse{  \lengthtest{\y pt < 2pt}  }
           {
                \draw [black]   (\x,\y,\z)-- (\x,\y+1,\z);
           }
           {
           }
           \ifthenelse{  \lengthtest{\z pt < 2pt}  }
           {
                \draw [black]   (\x,\y,\z) -- (\x,\y,\z+1);
           }
           {
           }
         \node[scale=1] at (0,0,0) {$\bullet$};
            \node[scale=0.9] at (1,2,1) {$\blacktriangle$};
          \node[scale=0.9,red] at (1,0,0) {$\bigstar$};
             \node[scale=0.9,red] at (0,2,0) {$\bigstar$};
              \node[scale=0.9,red] at (0,0,1) {$\bigstar$};
               \node[scale=0.9,blue] at (1,2,0) {$\blacksquare$};
                \node[scale=0.9,blue] at (1,0,1) {$\blacksquare$};
                  \node[scale=0.9,blue] at (0,2,1) {$\blacksquare$};               
}

\end{tikzpicture}
\end{center}
\caption{For the interval modules $V_{I_j}$ for $j=0,1,2,3$ given in the proof of Theorem \ref{thm:no pictorial formula}, the figures above depict the multigraded Betti numbers for $V_{I_j}$ in order from left to right. 
In these illustrations, a black dot at $p$ indicates that $\beta_0(V_{I_j})(p)=1$, a red star at $p$ indicates that $\beta_1(V_{I_j})(p)=1$, a blue square at $p$ indicates that $\beta_2(V_{I_j})(p)=1$, and a black triangle at $p$ indicates that $\beta_3(V_{I_j})(p)=1$. For $i\geq 3$, $j\in\{0,1,2,3\}$, and $p\in \Zb^3$, $\beta_i(V_{I_j})(p)=0$. }\label{fig:3D}
\end{figure}

\section{Conclusions}\label{sec:discussion}

The formula in Theorem \ref{thm:pictorial formula} for computing the bigraded Betti numbers reinforces the fact that the ($\inter$-)generalized persistence diagram and the \emph{interval decomposable approximation} by Asashiba et al. (Remark \ref{rem:asashiba}) are a proxy for the ``barcode" of $M$ in a novel way. Some open questions follow.
\begin{enumerate}[label=(\roman*)]
    \item Note that when $M$ is a finitely generated $\Zb^2$-module, $\dgm(M)$ can recover $\intdgm(M)$ by construction while $\intdgm(M)$ may not be able to recover $\dgm(M)$. However, if $M$ is interval decomposable, then both $\dgm(M)$ and $\intdgm(M)$ are equivalent to the barcode of $M$ by Theorem \ref{thm:GPD is barcode proxy}. Are there other settings in which $\intdgm(M)$ can recover $\dgm(M)$?
    \item \new{How can we utilize Theorem \ref{thm:pictorial formula} alongside efficient algorithms for computing the bigraded Betti numbers  \cite{lesnick2019computing, kerber2021fast}, as a means to estimate or calculate the generalized persistence diagram of a 2-parameter persistence module (cf. Remark \ref{rem:meaning})?}
\end{enumerate}

\bibliographystyle{abbrv}
\bibliography{biblio}

\begin{thebibliography}{10}

\bibitem{asashiba2018interval}
H.~Asashiba, M.~Buchet, E.~G. Escolar, K.~Nakashima, and M.~Yoshiwaki.
\newblock On interval decomposability of 2d persistence modules.
\newblock {\em Computational Geometry}, 105-106:101879, 2022.

\bibitem{asashiba2019approximation}
H.~Asashiba, E.~G. Escolar, K.~Nakashima, and M.~Yoshiwaki.
\newblock On approximation of $2$ d persistence modules by
  interval-decomposables.
\newblock {\em arXiv preprint arXiv:1911.01637}, 2019.

\bibitem{azumaya1950corrections}
G.~Azumaya et~al.
\newblock Corrections and supplementaries to my paper concerning
  {K}rull-{R}emak-{S}chmidt's theorem.
\newblock {\em Nagoya Math. J.}, 1:117--124, 1950.

\bibitem{bauer2020cotorsion}
U.~Bauer, M.~B. Botnan, S.~Oppermann, and J.~Steen.
\newblock Cotorsion torsion triples and the representation theory of filtered
  hierarchical clustering.
\newblock {\em Advances in Mathematics}, 369:107171, 2020.

\bibitem{bender1975applications}
E.~A. Bender and J.~R. Goldman.
\newblock On the applications of {M}{\"o}bius inversion in combinatorial
  analysis.
\newblock {\em The American Mathematical Monthly}, 82(8):789--803, 1975.

\bibitem{blanchette2021homological}
B.~Blanchette, T.~Br{\"u}stle, and E.~J. Hanson.
\newblock Homological approximations in persistence theory.
\newblock {\em arXiv preprint arXiv:2112.07632}, 2021.

\bibitem{botnan2020decomposition}
M.~Botnan and W.~Crawley-Boevey.
\newblock Decomposition of persistence modules.
\newblock {\em Proceedings of the American Mathematical Society},
  148(11):4581--4596, 2020.

\bibitem{botnan2018algebraic}
M.~Botnan and M.~Lesnick.
\newblock Algebraic stability of zigzag persistence modules.
\newblock {\em Algebr. Geom. Topol.}, 18(6):3133--3204, 2018.

\bibitem{botnan2020rectangle}
M.~B. Botnan, V.~Lebovici, and S.~Oudot.
\newblock {On Rectangle-Decomposable 2-Parameter Persistence Modules}.
\newblock In {\em 36th International Symposium on Computational Geometry (SoCG
  2020)}, volume 164 of {\em Leibniz International Proceedings in Informatics
  (LIPIcs)}, pages 22:1--22:16, Dagstuhl, Germany, 2020. Schloss
  Dagstuhl--Leibniz-Zentrum f{\"u}r Informatik.

\bibitem{botnan2021signed}
M.~B. Botnan, S.~Oppermann, and S.~Oudot.
\newblock Signed barcodes for multi-parameter persistence via rank
  decompositions and rank-exact resolutions.
\newblock {\em arXiv preprint arXiv:2107.06800}, 2021.

\bibitem{cai2020elder}
C.~Cai, W.~Kim, F.~M{\'e}moli, and Y.~Wang.
\newblock Elder-rule-staircodes for augmented metric spaces.
\newblock {\em SIAM Journal on Applied Algebra and Geometry}, 5(3):417--454,
  2021.

\bibitem{carlsson2009topology}
G.~Carlsson.
\newblock Topology and data.
\newblock {\em Bulletin of the American Mathematical Society}, 46(2):255--308,
  2009.

\bibitem{carlsson2010zigzag}
G.~Carlsson and V.~De~Silva.
\newblock Zigzag persistence.
\newblock {\em Foundations of computational mathematics}, 10(4):367--405, 2010.

\bibitem{carlsson2009zigzag}
G.~Carlsson, V.~De~Silva, and D.~Morozov.
\newblock Zigzag persistent homology and real-valued functions.
\newblock In {\em Proceedings of the twenty-fifth annual symposium on
  Computational geometry}, pages 247--256, 2009.

\bibitem{carlsson2010multiparameter}
G.~Carlsson and F.~M{\'e}moli.
\newblock Multiparameter hierarchical clustering methods.
\newblock In {\em Classification as a Tool for Research}, pages 63--70.
  Springer, Heidelberg, 2010.

\bibitem{carlsson2009theory}
G.~Carlsson and A.~Zomorodian.
\newblock The theory of multidimensional persistence.
\newblock {\em Discrete \& Computational Geometry}, 42(1):71--93, 2009.

\bibitem{carlsson2005persistence}
G.~Carlsson, A.~Zomorodian, A.~Collins, and L.~J. Guibas.
\newblock Persistence barcodes for shapes.
\newblock {\em International Journal of Shape Modeling}, 11(02):149--187, 2005.

\bibitem{chacholski2017combinatorial}
W.~Chach{\'o}lski, M.~Scolamiero, and F.~Vaccarino.
\newblock Combinatorial presentation of multidimensional persistent homology.
\newblock {\em Journal of Pure and Applied Algebra}, 221(5):1055--1075, 2017.

\bibitem{chambers2018persistent}
E.~Chambers and D.~Letscher.
\newblock Persistent homology over directed acyclic graphs.
\newblock In {\em Research in Computational Topology}, pages 11--32. Springer,
  Switzerland, 2018.

\bibitem{cochoy2020decomposition}
J.~Cochoy and S.~Oudot.
\newblock Decomposition of exact pfd persistence bimodules.
\newblock {\em Discrete \& Computational Geometry}, 63(2):255--293, 2020.

\bibitem{cohen2007stability}
D.~Cohen-Steiner, H.~Edelsbrunner, and J.~Harer.
\newblock Stability of persistence diagrams.
\newblock {\em Discrete \& computational geometry}, 37(1):103--120, 2007.

\bibitem{crawley2015decomposition}
W.~Crawley-Boevey.
\newblock Decomposition of pointwise finite-dimensional persistence modules.
\newblock {\em Journal of Algebra and its Applications}, 14(05):1550066, 2015.

\bibitem{derksen2005quiver}
H.~Derksen and J.~Weyman.
\newblock Quiver representations.
\newblock {\em Notices of the AMS}, 52(2):200--206, 2005.

\bibitem{dey2021computing}
T.~K. Dey, W.~Kim, and F.~M\'{e}moli.
\newblock {Computing Generalized Rank Invariant for 2-Parameter Persistence
  Modules via Zigzag Persistence and Its Applications}.
\newblock In {\em 38th International Symposium on Computational Geometry (SoCG
  2022)}, volume 224 of {\em Leibniz International Proceedings in Informatics
  (LIPIcs)}, pages 34:1--34:17, Dagstuhl, Germany, 2022. Schloss Dagstuhl --
  Leibniz-Zentrum f{\"u}r Informatik.

\bibitem{dey2018computing}
T.~K. Dey and C.~Xin.
\newblock {Computing Bottleneck Distance for 2-D Interval Decomposable
  Modules}.
\newblock In {\em 34th International Symposium on Computational Geometry (SoCG
  2018)}, volume~99 of {\em Leibniz International Proceedings in Informatics
  (LIPIcs)}, pages 32:1--32:15. Schloss Dagstuhl--Leibniz-Zentrum fuer
  Informatik, 2018.

\bibitem{dey2021rectangular}
T.~K. Dey and C.~Xin.
\newblock Rectangular approximation and stability of $2 $-parameter persistence
  modules.
\newblock {\em arXiv preprint arXiv:2108.07429}, 2021.

\bibitem{dey2019generalized}
T.~K. Dey and C.~Xin.
\newblock Generalized persistence algorithm for decomposing multiparameter
  persistence modules.
\newblock {\em Journal of Applied and Computational Topology}, pages 1--52,
  2022.

\bibitem{edelsbrunner2010computational}
H.~Edelsbrunner and J.~Harer.
\newblock {\em Computational topology: an introduction}.
\newblock American Mathematical Soc., 2010.

\bibitem{edelsbrunner2000topological}
H.~Edelsbrunner, D.~Letscher, and A.~Zomorodian.
\newblock Topological persistence and simplification.
\newblock In {\em Proceedings 41st annual symposium on foundations of computer
  science}, pages 454--463. IEEE, 2000.

\bibitem{eisenbudbetti}
D.~Eisenbud.
\newblock {\em The Geometry of Syzygies: A Second Course in Algebraic Geometry
  and Commutative Algebra}.
\newblock Springer, New York, 2002.

\bibitem{escolar2016persistence}
E.~G. Escolar and Y.~Hiraoka.
\newblock Persistence modules on commutative ladders of finite type.
\newblock {\em Discrete \& Computational Geometry}, 55(1):100--157, 2016.

\bibitem{gabrielsthm}
P.~Gabriel.
\newblock Unzerlegbare darstellungen i.
\newblock {\em Manuscripta Mathematica}, pages 71--103, 1972.

\bibitem{harrington2019stratifying}
H.~A. Harrington, N.~Otter, H.~Schenck, and U.~Tillmann.
\newblock Stratifying multiparameter persistent homology.
\newblock {\em SIAM Journal on Applied Algebra and Geometry}, 3(3):439--471,
  2019.

\bibitem{hilbertsyzygy}
D.~Hilbert.
\newblock \"{U}uber die theorie von algebraischen formen.
\newblock {\em Mathematische Annalen}, pages 473--534, 1890.

\bibitem{hiraoka2016hierarchical}
Y.~Hiraoka, T.~Nakamura, A.~Hirata, E.~G. Escolar, K.~Matsue, and Y.~Nishiura.
\newblock Hierarchical structures of amorphous solids characterized by
  persistent homology.
\newblock {\em Proceedings of the National Academy of Sciences},
  113(26):7035--7040, 2016.

\bibitem{keller2018persistent}
B.~Keller, M.~Lesnick, and T.~L. Willke.
\newblock Persistent homology for virtual screening.
\newblock {\em ChemRxiv.6969260.v3}, 2018.

\bibitem{kerber2021fast}
M.~Kerber and A.~Rolle.
\newblock Fast minimal presentations of bi-graded persistence modules.
\newblock In {\em 2021 Proceedings of the Workshop on Algorithm Engineering and
  Experiments (ALENEX)}, pages 207--220. SIAM, 2021.

\bibitem{kim2018generalized}
W.~Kim and F.~M{\'e}moli.
\newblock Generalized persistence diagrams for persistence modules over posets.
\newblock {\em Journal of Applied and Computational Topology}, 5(4):533--581,
  2021.

\bibitem{knudson2007refinement}
K.~P. Knudson.
\newblock A refinement of multi-dimensional persistence.
\newblock {\em Homology Homotopy Appl.}, 10(1):259--281, 2008.

\bibitem{landi1997new}
C.~Landi and P.~Frosini.
\newblock New pseudodistances for the size function space.
\newblock In {\em Vision Geometry VI}, volume 3168, pages 52--60. International
  Society for Optics and Photonics, 1997.

\bibitem{lesnick2015theory}
M.~Lesnick.
\newblock The theory of the interleaving distance on multidimensional
  persistence modules.
\newblock {\em Foundations of Computational Mathematics}, 15(3):613--650, 2015.

\bibitem{lesnick2015interactive}
M.~Lesnick and M.~Wright.
\newblock Interactive visualization of 2-d persistence modules.
\newblock {\em arXiv preprint arXiv:1512.00180}, 2015.

\bibitem{lesnick2019computing}
M.~Lesnick and M.~Wright.
\newblock Computing minimal presentations and bigraded betti numbers of
  2-parameter persistent homology.
\newblock {\em SIAM Journal on Applied Algebra and Geometry}, 6(2):267--298,
  2022.

\bibitem{mac2013categories}
S.~Mac~Lane.
\newblock {\em Categories for the working mathematician}, volume~5.
\newblock Springer Science \& Business Media, New York, 2013.

\bibitem{mccleary2018bottleneck}
A.~McCleary and A.~Patel.
\newblock Bottleneck stability for generalized persistence diagrams.
\newblock {\em Proceedings of the American Mathematical Society},
  148(7):3149--3161, 2020.

\bibitem{mccleary2020edit}
A.~McCleary and A.~Patel.
\newblock Edit distance and persistence diagrams over lattices.
\newblock {\em SIAM Journal on Applied Algebra and Geometry}, 6(2):134--155,
  2022.

\bibitem{miller2020homological}
E.~Miller.
\newblock Homological algebra of modules over posets.
\newblock {\em arXiv preprint arXiv:2008.00063}, 2020.

\bibitem{miller2005combinatorial}
E.~Miller and B.~Sturmfels.
\newblock {\em Combinatorial commutative algebra}, volume 227.
\newblock Springer Science \& Business Media, New York, 2005.

\bibitem{moore2020subquivers}
S.~Moore.
\newblock A combinatorial formula for the bigraded betti numbers.
\newblock {\em arXiv preprint arXiv:2004.02239}, 2020.

\bibitem{moore2022structure}
S.~Moore.
\newblock {\em On the structure of multiparameter persistence modules}.
\newblock PhD thesis, University of North Carolina at Chapel Hill, 2022.

\bibitem{patel2018generalized}
A.~Patel.
\newblock Generalized persistence diagrams.
\newblock {\em Journal of Applied and Computational Topology}, 1(3):397--419,
  2018.

\bibitem{rota1964foundations}
G.-C. Rota.
\newblock On the foundations of combinatorial theory i. theory of m{\"o}bius
  functions.
\newblock {\em Zeitschrift f{\"u}r Wahrscheinlichkeitstheorie und verwandte
  Gebiete}, 2(4):340--368, 1964.

\bibitem{scolamiero2017multidimensional}
M.~Scolamiero, W.~Chach{\'o}lski, A.~Lundman, R.~Ramanujam, and S.~{\"O}berg.
\newblock Multidimensional persistence and noise.
\newblock {\em Foundations of Computational Mathematics}, 17(6):1367--1406,
  2017.

\bibitem{thomas2019invariants}
A.~L. Thomas.
\newblock {\em Invariants and metrics for multiparameter persistent homology}.
\newblock PhD thesis, Duke University, 2019.

\bibitem{vipond2020multiparameter}
O.~Vipond.
\newblock Multiparameter persistence landscapes.
\newblock {\em J. Mach. Learn. Res.}, 21:61--1, 2020.

\end{thebibliography}

\appendix
\section{Appendix}

\paragraph{Types of corner points.} We define the $0^\mathrm{th}$, $1^\mathrm{st}$ and $2^\mathrm{nd}$ type corner points that are depicted in Fig. \ref{fig:corner points in general}. Given any $A\subset \Rb^2$, let $\one_A:\Rb^2\rightarrow \{0,1\}$ be the indicator function of $A$, i.e.
$\one_A(p)=1$ if $p\in 
A$ and zero otherwise.

\begin{definition}\label{def:types of corners} Let $I\in \con(\Zb^2)$ and let $I^+:=\bigcup_{(p_1,p_2)\in I}[p_1,p_1+1)\times [p_2,p_2+1)\subset \Rb^2$. Fix $\ba\in \Rb^2$. This $\ba$ is a \textbf{0-th type corner point} of $I^+$ if \[\one_{I^+}(\ba)=1, \ \ \  \lim_{\eps\rightarrow0+}\one_{I^+}(\ba-(\eps,0))=\lim_{\eps\rightarrow0+}\one_{I^+}(\ba-(0,\eps))=\lim_{\eps\rightarrow0+}\one_{I^+}(\ba-(\eps,\eps))=0.\] The point $\ba$ is a \textbf{1-st type corner point with multiplicity $k$ ($k=1,2$)} of $I^+$ if one of the following two conditions holds:
\begin{enumerate}[label=(\roman*)]
    \item ($k=1$) Either the following is evaluated to be -1
    \begin{equation}\label{eq:1st corner determinant}
        \one_{I^+}(\ba)- \lim_{\eps\rightarrow0+}\one_{I^+}(\ba-(\eps,0))-\lim_{\eps\rightarrow0+}\one_{I^+}(\ba-(0,\eps))+\lim_{\eps\rightarrow0+}\one_{I^+}(\ba-(\eps,\eps))
    \end{equation} 
    (cf. (i)-(iv) in the panel corresponding to the $1^{\mathrm{st}}$ type in Figure \ref{fig:corner points in general}), or the following holds: \[\one_{I^+}(\ba)=\lim_{\eps\rightarrow0+}\one_{I^+}(\ba-(\eps,\eps))=1\ \  \mbox{ and }\ \  \lim_{\eps\rightarrow0+}\one_{I^+}(\ba-(\eps,0))\neq  \lim_{\eps\rightarrow0+}\one_{I^+}(\ba-(0,\eps))\]
     (cf. (v) and (vi) in the panel corresponding to the $1^{\mathrm{st}}$ type in Figure \ref{fig:corner points in general}).
    
    \item \textbf{($k=2$)} The formula given in (\ref{eq:1st corner determinant}) is evaluated to be -2 (cf. the point $p$ in Figure \ref{fig:corner points in general}). 
\end{enumerate}

The point $\ba$ is a \textbf{2-nd type corner point} of $I^+$ if \[\one_{I^+}(\ba)= \lim_{\eps\rightarrow0+}\one_{I^+}(\ba-(\eps,0))=\lim_{\eps\rightarrow0+}\one_{I^+}(\ba-(0,\eps))=0, \ \mbox{ and } \lim_{\eps\rightarrow0+}\one_{I^+}(\ba-(\eps,\eps))=1.\]
\end{definition}
Definition \ref{def:types of corners} is closely related to the \emph{differential} of an interval introduced in \cite{dey2018computing}.

\paragraph{The generalized persistence diagram is more discriminative than the $\inter$-generalized persistence diagram.}

We provide a pair of persistence modules that are distinguished by their generalized rank invariants (and hence by their generalized persistence diagrams) but have the same $\inter$-generalized rank invariant (and hence the same $\inter$-generalized persistence diagram).

\begin{example}\label{ex:GPD_vs_IGPD} Let $M,N:[2]^2\rightarrow \vect$ and $J\in \con([2]^2)$ be defined as in Figure \ref{fig:GPD_vs_IGPD}. Then, $\rk(M)(J)=1$ whereas $\rk(N)(J)=0$; this directly follows from Theorem \ref{thm:rkequalsintervals}. Note also that for all $I\supsetneq J$ in $\con([2]^2)$, we have that $\rk(M)(I)=\rk(N)(I)=0$. This implies, by equations (\ref{eq:induction}), (\ref{eq:GPD}), (\ref{eq:IGPD}), that $\dgm(M)(J)=1\neq 0=\dgm(N)(J)$.

Invoking Theorem \ref{thm:rkequalsintervals} again, one can check that $\rk(M)(I)=\rk(N)(I)$ for all $I\in \inter([2]^2)$, i.e. $\intrk(M)=\intrk(N)$ and thus $\intdgm(M)=\intdgm(N)$.
\end{example}

\begin{figure}
    \centering
    \includegraphics[width=0.7\textwidth]{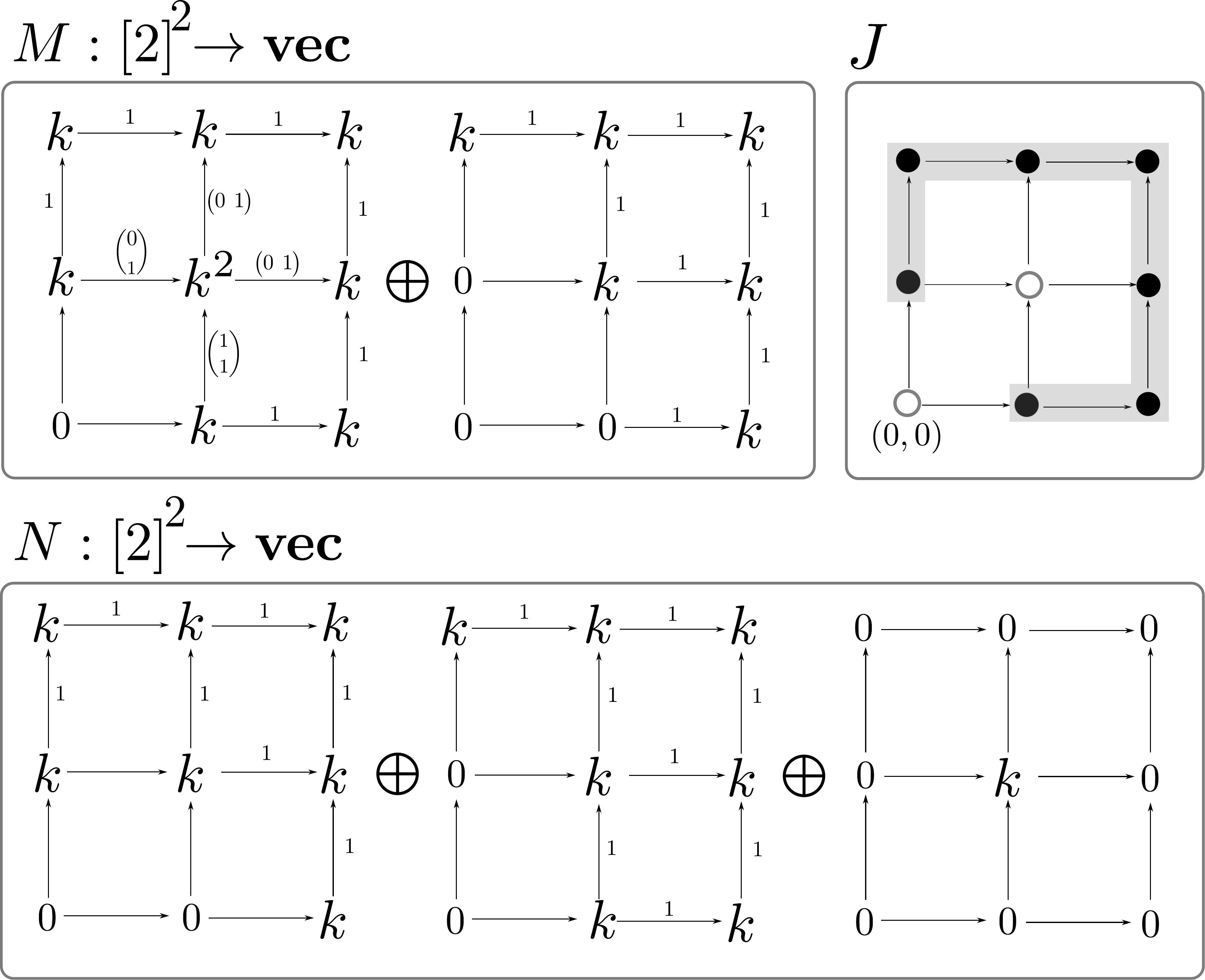}
    \caption{An illustration for Example \ref{ex:GPD_vs_IGPD} illustrating that, in general, $\dgm(M)$ is a stronger invariant than $\intdgm(M)$}.
    \label{fig:GPD_vs_IGPD}
\end{figure}

\paragraph{Limits and colimits.}\label{sec:limits}
 We recall the notions of limit and colimit \cite[{Chapter \rom{5}}]{mac2013categories}. In what follows, $I$ stands for a small category, i.e. $I$ has a set of objects and a set of morphisms. Let $\Ccal$ be any category.
\begin{definition}[Cone]\label{def:cone} Let $F:I\rightarrow \Ccal$ be a functor. A \emph{cone} over $F$ is a pair $\left(L,(\pi_x)_{x\in \ob(I)}\right)$ consisting of an object $L$ in $\Ccal$ and a collection $(\pi_x)_{x\in \ob(I)}$ of morphisms $\pi_x:L \rightarrow F(x)$ that commute with the arrows in the diagram of $F$, i.e. if $g:x\rightarrow y$ is a morphism in $I$, then $\pi_y= F(g)\circ \pi_x$ in $\Ccal$. Equivalently, the diagram below commutes:
\end{definition}
\[\begin{tikzcd}F(x)\arrow{rr}{F(g)}&&F(y)\\
& L \arrow{lu}{\pi_x} \arrow{ru}[swap]{\pi_y}\end{tikzcd}\]

A limit of $F:I\rightarrow \Ccal$ is a terminal object in the collection of all cones over $F$:

\begin{definition}[Limit]\label{def:limit} Let $F:I\rightarrow \Ccal$ be a functor. A \emph{limit} of $F$ is a cone over $F$, denoted by $\left(\varprojlim F,\ (\pi_x)_{x\in \ob(I)} \right)$ or simply $\varprojlim F$, with the following \emph{terminal} property: If there is another cone $\left(L',(\pi'_x)_{x\in \ob(I)} \right)$ of $F$, then there is a \emph{unique} morphism $u:L'\rightarrow \varprojlim F$ such that $\pi_x'=\pi_x\circ u$ for all $x\in \ob(I)$.
\end{definition}

It is possible that a functor does not have a limit at all. However, if a functor does have a limit then the terminal property of the limit guarantees its uniqueness up to isomorphism. For this reason, we sometimes refer to a limit as \emph{the} limit of a functor. When $I$ is a finite category and $\Ccal=\vect$, any functor $F:I\rightarrow \vect$ admits a limit in $\vect$.

Cocones and colimits are defined in a dual manner:
\begin{definition}[Cocone]\label{def:cocone} Let $F:I\rightarrow \Ccal$ be a functor. A \emph{cocone} over $F$ is a pair $\left(C,(i_x)_{x\in \ob(I)}\right)$ consisting of an object $C$ in $\Ccal$ and a collection $(i_x)_{x\in \ob(I)}$ of morphisms $i_x:F(x)\rightarrow C$ that commute with the arrows in the diagram of $F$, i.e. if $g:x\rightarrow y$ is a morphism in $I$, then $i_x= i_y\circ F(g)$ in $\Ccal$, i.e. the diagram below commutes.
\end{definition}
\[\begin{tikzcd}
&C\\
F(x)\arrow{ru}{i_x}\arrow{rr}[swap]{F(g)}&&F(y)\arrow{lu}[swap]{i_y}
\end{tikzcd}\]

A colimit of a functor $F:I\rightarrow \Ccal$ is an initial object in the collection of cocones over $F$:

\begin{definition}[Colimit]\label{def:colimit}Let $F:I\rightarrow \Ccal$ be a functor. A \emph{colimit} of $F$ is a cocone, denoted by $\left(\varinjlim F,\ (i_x)_{x\in \ob(I)}\right)$ or simply $\varinjlim F$, with the following \emph{initial} property: If there is another cocone $\left(C', (i'_x)_{x\in \ob(I)}\right)$ of $F$, then there is a \emph{unique} morphism $u:\varinjlim F\rightarrow C'$  such that $i'_x=u\circ i_x$ for all $x\in \ob(I)$.
\end{definition}

It is possible that a functor does not have a colimit at all. However, if a functor does have a colimit then the initial property of the colimit guarantees its uniqueness up to isomorphism. For this reason, we sometimes refer to a colimit as \emph{the} colimit of a functor. When $I$ is a finite category and $\Ccal=\vect$, any functor $F:I\rightarrow \vect$ admits a colimit in $\vect$.

\paragraph{Multirank invariant.} 
We review the notion of the \emph{multirank invariant} for a persistence module \cite{thomas2019invariants}, which is a natural generalization of the rank invariant and differs from the generalized rank invariant. Then, we demonstrate that the multirank invariant of a zigzag module $M$ with a length of 3 completely determines the isomorphism type of $M$. This fact, along with Corollary \ref{cor:zz barcode determine bettis}, implies that the multirank invariant of any $\Zb^2$-module $N$ determines the bigraded Betti numbers of $N$. 

Let $\Pb$ be a poset. For a $\Pb$-module $M$ and any $s,t\in \Pb$, let us define the map

\begin{equation}\label{eq:map from s to t}M(s)\rightarrow M(t)=\begin{cases}
\varphi_M(s, t),&\mbox{if $s\leq t$}\\0,&\mbox{otherwise.}
\end{cases}
\end{equation}

\begin{definition}\label{def:multirank} For a $\Pb$-module $M$ and finite subsets $S, T \subset
\Pb$, the \textbf{multirank} from $S$ to $T$ is defined as 
\[\mrank_M(S,T):=\rank \left(\bigoplus_{s\in S}M(s)\rightarrow \bigoplus_{t\in T}M(t)\right),\]
where the map $\bigoplus_{s\in S}M(s)\rightarrow \bigoplus_{t\in T}M(t)$ is canonically defined by the map given in Equation (\ref{eq:map from s to t}) for each pair of $s\in S$ and $t\in T$. The \textbf{multirank invariant} of $M$ is the map that sends every pair of finite subsets $S,T\subset \Pb$ to $\mrank_M(S,T)$.
\end{definition}

A list of useful properties of the multirank invariant follows.
\begin{remark}\label{rem:properties of multirank} Let $M$ be any $\Pb$-module.
\begin{enumerate}[label=(\roman*)]
\item The multirank invariant subsumes the rank invariant: If $s,t\in \Pb$ with $s\leq t$, then for $S=\{s\}$ and $T=\{t\}$, the rank of $\varphi_{M}(s, t)$ coincides with $\mrk_M(S,T)$. \label{item:properties of multirank0}

\item For any $S,T\subset \Pb$ such that there is no pair $(s,t)\in S\times T$ with $s\leq t$ in $\Pb$, we have that 
\[\mrk_M(S,T)=0.\]\label{item:properties of multirank1}
\item For any $S\subset \Pb$ and any disjoint pair $T_1,T_2\subset \Pb$, \[\mrk_M(S,T_1\cup T_2)=\mrk_M(S,T_1)+\mrk_M(S, T_2).\]\label{item:properties of multirank2}
\item For any $T\subset S \subset \Pb$,
\[\mrk_M(S,T)=\sum_{t\in T}\dim M_t.\]\label{item:properties of multirank3}
\end{enumerate}

Let $N$ be another $\Pb$-module.
\begin{enumerate}[label=(\roman*),resume]
\item For any $S,T\subset \Pb$, \[\mrk_{M\oplus N}(S,T)=\mrk_M(S,T)+\mrk_N(S, T).\]\label{item:properties of multirank4}
\item  If $\dim(M_p)=\dim(N_p)$ for each $p\in \Pb$, then item \ref{item:properties of multirank3} implies that, whenever $T\subset S\subset \Pb$, 
$\mrk_M(S,T)=\mrk_N(S,T).$\label{item:properties of multirank5}
\item By Items \ref{item:properties of multirank2} and \ref{item:properties of multirank5}, if $\dim(M_p)=\dim(N_p)$ for every $p\in \Pb$, then the multirank invariants of $M$ and $N$ are identical if and only if for every \emph{disjoint} pair of subsets $S,T\in \Pb$,  
$\mrk_M(S,T)= \mrk_N(S,T)$.\label{item:properties of multirank6}
\end{enumerate}
\end{remark}

The multirank invariant is a complete invariant for zigzag modules of length 3:

\begin{proposition}\label{prop:multirank determines length-3 zigzag} Let $\Pb$ be any zigzag poset of length 3, and let $M,N$ be any $\Pb$-modules. If $M$ and $N$ have the same multirank invariant, then $M$ and $N$ are isomorphic.
\end{proposition}

\begin{proof}  
Assume that $\Pb=\{\bullet_1<\bullet_2<\bullet_3\}$. In this case, the rank invariant of $M$ uniquely determines the isomorphism type of $M$. Since the rank invariant is subsumed by the multirank invariant (Remark \ref{rem:properties of multirank} \ref{item:properties of multirank0}), the claim follows.

Next, assume that $\Pb=\{\bullet_1> \bullet_2 <\bullet_3\}$. Since every $\Pb$-module $M$ is interval decomposable (Theorem \ref{thm:interval decomposable}), it suffices to show that the barcode of $M$ can be extracted from the multirank invariant of $M$. The multirank of the interval modules $V_I:\Pb\rightarrow \vect$ for different pairs of $S,T\subset \Pb$ and for different intervals $I$ are given in the table below.

\begin{center}
\begin{tabular}{|l|l|*{6}{c|}}\hline

$(S,T)$
&$\{\bullet_1\}$&$\{\bullet_2\}$&$\{\bullet_3\}$
&$\{\bullet_1,\bullet_2\}$&$\{\bullet_2,\bullet_3\}$&$\{\bullet_1,\bullet_2,\bullet_3\}$\\\hline\hline
$\left(\{\bullet_1\}, \{\bullet_1\}\right)$ &1&&&1&&1\\\hline
$\left(\{\bullet_2\}, \{\bullet_2\}\right)$ &&1&&1&1&1\\\hline
$\left(\{\bullet_3\}, \{\bullet_3\}\right)$ &&&1&&1&1\\\hline
$\left(\{\bullet_2\}, \{\bullet_3\}\right)$ &&&&&1&1\\\hline
$\left(\{\bullet_2\}, \{\bullet_1\}\right)$ &&&&1&&1\\\hline
$\left(\{\bullet_2\}, \{\bullet_1,\bullet_3\}\right)$ &&&&1&1&1\\\hline
\end{tabular}
\end{center}
For example, the entry 1 in the second row and fifth column indicates that for $I=\{\bullet_1,\bullet_2\}$, $\mrk_{V_{I}}\left(\{\bullet_1\}, \{\bullet_1\}\right)=1$.
By Remark \ref{rem:properties of multirank} \ref{item:properties of multirank4}, the table above allows us to extract the barcode of $M$ as follows:
\begin{enumerate}[label=(\roman*)]
    \item From Rows 6 and 7, the multiplicity of $\{\bullet_2,\bullet_3\}$ in $\barc(M)$ equals \[\mrk_M(\{\bullet_2\},\{\bullet_1,\bullet_3\})-\mrk_M(\{\bullet_2\},\{\bullet_1\}).
    \]\label{item:23}
    \item From Rows 5 and 7, the multiplicity of $\{\bullet_1,\bullet_2\}$ in $\barc(M)$ equals \[\mrk_M(\{\bullet_2\},\{\bullet_1,\bullet_3\})-\mrk_M(\{\bullet_2\},\{\bullet_3\}).
    \]\label{item:12}
    \item From Rows 5, 6, and 7, the multiplicity of $\{\bullet_1,\bullet_2,\bullet_3\}$ in $\barc(M)$ equals
    \[\mrk_M(\{\bullet_2\},\{\bullet_3\})+\mrk_M(\{\bullet_2\},\{\bullet_1\})-\mrk_M(\{\bullet_2\},\{\bullet_1,\bullet_3\}).
    \]\label{item:123}
    \item The multiplicity of $\{\bullet_1\}$ in $\barc(M)$ is equal to $\mrk_M(\{\bullet_1\},\{\bullet_1\})$ (cf. Row 2) minus the sum of the multiplicities of $\{\bullet_1,\bullet_2\}$ and $\{\bullet_1,\bullet_2,\bullet_3\}$ found in the previous two items. The multiplicities of $\{\bullet_2\}$ and $\{\bullet_3\}$ in $\barc(M)$ can be computed in a similar way.
\end{enumerate}
Similarly, when assuming $\Pb=\{\bullet_1<\bullet_2>\bullet_3\}$, the barcode of any $\Pb$-module $N$ can be extracted from the multirank invariant of $N$.
\end{proof}

By Corollary \ref{cor:zz barcode determine bettis} and Proposition \ref{prop:multirank determines length-3 zigzag}, we have:

\begin{corollary}\label{cor:multirank determines bettis} The multirank invariant of a finitely generated $\Zb^2$-module $M$ determines the bigraded Betti numbers of $M$.
\end{corollary}

Since the bigraded Betti numbers do not determine the standard rank invariant of a 2-parameter persistence module \cite[Section 1.6]{lesnick2015interactive}, the fact that the multirank invariant subsumes the standard rank invariant (Remark \ref{rem:properties of multirank} \ref{item:properties of multirank0}) implies that the converse of the previous corollary does not hold.

\paragraph{Comparison between the multirank invariant and the generalized rank invariant.}
We show that neither the generalized rank invariant nor the multirank invariant is a strictly stronger invariant than the other. We do this by providing two pairs of persistence modules: The first pair is distinguishable by their multirank invariants but not by their generalized rank invariants. The second pair exhibits the opposite scenario.

In what follows, let $\Pb:=\{a,b,c,d\}$ be the poset equipped with the partial order $\leq:=\{(b,a),(c,a),(d,a)\}\subset \Pb\times \Pb$. The Hasse diagram of $\Pb$ is given below.

\begin{center}
\begin{tikzcd}[sep=small]
  & a \arrow[d, dash] \arrow[ld, dash] \arrow[rd, dash]\\
  b  & c & d
 \end{tikzcd}
 \end{center}
\begin{example}\label{ex:identical GPD, different mrk}  A $\Pb$-module $M$ is given by
\begin{center}
\begin{tikzcd}
  & k^2   \\
  k \arrow[ru, hook,"i_1"] & k \arrow[u, hook,"i_2"]& k, \arrow[ul, hook,"i_1+i_2"']
 \end{tikzcd}
 \end{center}
\noindent where $i_1,i_2:k\rightarrow k^2$ are the canonical inclusions into the first factor and the second factor of $k^2$, respectively. 
Consider the interval modules $V_{\{a\}}$, $V_{\{a,b\}}$, $V_{\{a,c\}}$, $V_{\{a,d\}}$  over $\Pb$.

It is not hard to verify that the generalized rank invariants of the $\Pb$-modules $N_1:=M\oplus V_{\{a\}}$ and $N_2:=V_{\{a,b\}}\oplus V_{\{a,c\}} \oplus V_{\{a,d\}}$ are identical. However, for $S:=\{b,c,d\}$ and $T:=\{a\}$, we have that 
\[\mrk_{N_1}(S,T)=2\neq3=\mrk_{N_2}(S,T).
\]
\end{example}

\begin{example}\label{ex:identical mrk, different GPD}
Consider the $\Pb$-modules $L_1:=M\oplus V_{\Pb}$ and $L_2:=V_{\{a,b,c\}}\oplus V_{\{a,c,d\}}\oplus V_{\{a,b,d\}}$. By Theorem \ref{thm:rkequalsintervals}, we have that
\[\rk_{L_1}(\Pb)=1\neq 0=\rk_{L_2}(\Pb).\] 
On the other hand, we claim that for any pair $S,T\subset \Pb$,
\begin{equation}\label{eq:same mrk}\mrk_{L_1}(S,T)= \mrk_{L_2}(S,T).
\end{equation}
Since $\dim((L_1)_p)=\dim((L_2)_p)$ for each $p\in \Pb$, Remark \ref{rem:properties of multirank} \ref{item:properties of multirank6} implies that one needs to check Equation (\ref{eq:same mrk}) only for  disjoint nonempty sets $S,T\subset\Pb$. Let $S,T\subset \Pb$ be nonempty disjoint sets. Note that, unless $a$ belongs to $T$, the both sides of Equation (\ref{eq:same mrk}) are zero. Therefore, it is only necessary to verify Equation (\ref{eq:same mrk}) in the case where $S\subset\{b, c, d\}$ and $T = \{a\}$. This verification is straightforward.

\end{example}

\end{document}